\documentclass[11pt]{article}

\usepackage{epsfig,epsf,fancybox}
\usepackage{amsmath}
\usepackage{mathrsfs}
\usepackage{amssymb}
\usepackage{graphicx}
\usepackage{color}
\usepackage{multirow}
\usepackage{paralist}
\usepackage{verbatim}
\usepackage{galois}
\usepackage[algo2e,ruled,vlined,linesnumbered]{algorithm2e}
\usepackage{boxedminipage}
\usepackage{booktabs}
\usepackage{accents}
\usepackage{stmaryrd}
\usepackage[table]{xcolor}
\usepackage{hhline}

\usepackage{subfig}

\usepackage{natbib}

\usepackage{url}
\usepackage[colorlinks,linkcolor=magenta,citecolor=blue, pagebackref=true,backref=true]{hyperref}
\renewcommand*{\backrefalt}[4]{%
    \ifcase #1 \footnotesize{(Not cited.)}%
    \or        \footnotesize{(Cited on page~#2.)}%
    \else      \footnotesize{(Cited on pages~#2.)}%
    \fi}

\newtheorem{theorem}{Theorem}[section]

\newtheorem{lemma}[theorem]{Lemma}
\newtheorem{proposition}[theorem]{Proposition}

\newtheorem{assumption}[theorem]{Assumption}

\numberwithin{equation}{section}

\newcommand{\argmin}{\mathop{\rm argmin}}

\newcommand{\br}{\mathbb{R}}

\newcommand{\ba}{\begin{array}}
\newcommand{\ea}{\end{array}}

\newcommand{\FCal}{\mathcal{F}}

\begin{document}

\begin{center}

{\bf{\LARGE{Tight Complexity of Muon under Heavy-Tailed Noise}}}

\vspace*{.2in}
{\large{
\begin{tabular}{c}
Jiayu Zhang \and Tianyi Lin \\
\end{tabular}
}}

\vspace*{.2in}

\begin{tabular}{c}
Department of Industrial Engineering and Operations Research \\ 
Columbia University
\end{tabular}

\vspace*{.2in}

\today

\vspace*{.2in}

\begin{abstract}
A growing lesson from neural network optimization is that optimizer design should respect how the model is parametrized. The layerwise input-output structure of neural networks motivates Muon and Scion, whose updates also support hyperparameter transfer. At the same time, stochastic gradient noise in deep learning is often far from sub-Gaussian and may exhibit heavy tails. These crucial observations have shaped recent algorithmic principles for training neural networks, yet their joint theoretical consequences are underexplored. In particular, it remains unclear what dimension dependence is unavoidable for gradient-based methods when the problem class is defined using an input-output norm and the stochastic gradient noise is heavy-tailed, and whether higher-order smoothness can accelerate optimization. We study these questions through nonconvex smooth stochastic optimization over $\br^{m\times n}$ equipped with general norms and under $p^\textnormal{th}$-moment heavy-tailed noise, where the goal is to achieve an $\epsilon$-stationary point measured in the dual norm. Our first contribution is a \emph{dimension-dependent} lower bound: when $\max\{m,n\}$ is large enough, any method using only the stochastic gradients requires $\Omega(\min\{m, n\}\epsilon^{-\frac{3p-2}{p-1}})$ oracle calls for the problem class defined by the spectral norm, which is a common input-output matrix norm. We prove that a batched Muon method with the spectral norm can achieve the matching upper bound of $O(\min\{m, n\}\epsilon^{-\frac{3p-2}{p-1}})$. To exploit higher-order smoothness, we propose a transported Muon method and improve the bound to $O(\min\{m, n\}\epsilon^{-\frac{5p-3}{2p-2}})$ when the norm is spectral and the Hessian is Lipschitz. Finally, we incorporate practical heuristics into our transported method and evaluate it across multiple architectures and model sizes, demonstrating its flexibility and compatibility with neural network training.
\end{abstract}
\end{center}

\section{Introduction}\label{sec:intro}
Neural networks have evolved from multilayer perceptrons (MLP) trained by backpropagation~\citep{Rumelhart-1986-Learning} into a dominant paradigm of modern AI. Key milestones include mixtures of experts (MoE)~\citep{Jacobs-1991-Adaptive}, convolutional neural networks (CNN) for vision~\citep{Lecun-1998-Gradient}, recurrent neural networks (RNN) for sequential data~\citep{Hochreiter-1997-Long, Cho-2014-Learning}, and deep belief networks (DBN) that helped revive interest in deep architectures~\citep{Hinton-2006-Fast}. Since the breakthrough of AlexNet on ImageNet~\citep{Krizhevsky-2012-Imagenet}, deep networks have transformed vision, speech, language and generative AI, with residual networks (ResNet) enabling much deeper models~\citep{He-2016-Deep} and transformers enabling scalable sequence modeling~\citep{Vaswani-2017-Attention}. This progress has relied not only on data and hardware, but on nonconvex stochastic optimization methods, from SGD and momentum~\citep{Robbins-1951-Stochastic, Polyak-1964-Some, Nesterov-1983-Method, Sutskever-2013-Importance} to adaptive methods such as AdaGrad, RMSProp, Adam, and AdamW \citep{Duchi-2011-Adaptive, Tieleman-2012-Lecture, Kingma-2015-Adam, Loshchilov-2019-Decoupled}. Yet, these methods are designed around Euclidean and coordinatewise parameter geometries and standard stochastic gradient noise, rather than \emph{layerwise, heterogeneous geometry} and \emph{heavy-tailed stochastic gradient noise} observed in training modern neural networks~\citep{Glorot-2010-Understanding, Ioffe-2015-Batch, Simsekli-2019-Tail}.

Recent norm-based matrix optimizers have addressed this mismatch by scaling updates according to layerwise matrix geometry. The guiding principle is simple: for a neural-network layer, the update scale should reflect how the layer transforms its inputs, rather than only the Euclidean length of all entries in its weight matrix. Muon implements this idea through matrix-sign updates~\citep{Carlson-2015-Preconditioned, Jordan-2024-Muon, Bernstein-2024-Old}, while Scion provides a broader framework based on linear minimization oracles over input-output matrix norm balls~\citep{Pethick-2025-Training}. Thus, the use of input-output non-Euclidean norms is not merely a theoretical abstraction but is already present in modern optimizer design. These methods have also been connected to architecture-aware scaling, modular norm viewpoints, and hyperparameter transfer principles~\citep{Large-2024-Scalable, Bernstein-2025-Modular, Ioffe-2015-Batch, Ba-2016-Layer, Yang-2021-Feature, Yang-2021-Tuning, Yang-2023-Spectral}. Compared with coordinatewise rescaling used in Adam and its variants, Muon and Scion make update magnitudes more comparable across heterogeneous layers and have been recognized in large-scale neural network training~\citep{Liu-2025-Muon, Team-2025-Kimi}. This yields an important theoretical question: when the stochastic optimization problem is formulated in the input-output norm geometry and with heavy-tailed noise, what is the optimal convergence rate for gradient-based methods, and can Muon or its variants attain this rate?

A line of work~\citep{Zhang-2020-Adaptive, Cutkosky-2021-High, Hubler-2025-Gradient, Sun-2025-Revisiting, Chezhegov-2025-Clipping} has shown that Frobenius-norm-based matrix optimizers equipped with clipping and normalization are effective when stochastic gradients are heavy-tailed, while~\citet{Liu-2025-Nonconvex} demonstrated that only normalization is sufficient. 
However, the more natural problem class for spectral-norm-based matrix optimizers should be defined by the spectral norm and its dual norm under heavy-tailed noise~\citep{Bernstein-2025-Modular}. To our knowledge, it remains unclear what dimension dependence is unavoidable for gradient-based methods given this problem class, and whether higher-order smoothness can accelerate optimization.

In this paper, we study the above questions through nonconvex smooth stochastic optimization over $\mathbb{R}^{m\times n}$ with general norms under $p^\textnormal{th}$-moment heavy-tailed noise. Indeed, the general-norm formulation captures the geometry used by modern neural network optimizers, while the heavy-tailed noise model reflects the non-Gaussian stochasticity observed in training. Our analysis has three parts. For the lower bound, we combine the framework of \citet{Arjevani-2023-Lower}, the heavy-tailed hard instance of \citet{Liu-2025-Nonconvex}, and the sign-encoding technique of \citet{Jin-2026-Tight}, spreading independent random signs across all rows so that each oracle call reveals at most one row. The resulting bound holds for all randomized algorithms with only gradient access including Muon. We then show that mini-batch Muon~\citep{Jordan-2024-Muon,Pethick-2025-Training} matches the lower bound. The key ingredient is to control the momentum error in the dual norm. We further adapt the implicit gradient transport mechanism~\citep{Cutkosky-2020-Momentum} to accelerate Muon. The analysis is nontrivial since the updates are \textit{nonlinear}, discard gradient magnitudes, and must be controlled in \textit{non-Euclidean dual norms} under heavy-tailed noise. 

\paragraph{Contributions.} We focus on two central questions: (i) what dimension dependence is unavoidable in stochastic nonconvex matrix optimization beyond Frobenius geometry, and (ii) whether higher-order smoothness can be leveraged to accelerate Muon and its variants under heavy-tailed noise. Our contributions can be summarized as follows. 
\begin{enumerate}
\item We establish the sharp dimension dependence for stochastic nonconvex matrix optimization in spectral-norm geometry and under $p^{\textnormal{th}}$-moment heavy-tailed noise. In the space $(\mathbb{R}^{m\times n},\|\cdot\|_{\mathrm{op}})$, we show that when $\max\{m,n\}$ is large enough, any gradient-based method requires $\Omega(\min\{m, n\}\epsilon^{-\frac{3p-2}{p-1}})$ oracles to find an $\epsilon$-stationary point. We then show that a batched Muon method achieves the matching upper bound of $O(\min\{m, n\}\epsilon^{-\frac{3p-2}{p-1}})$. 
\item We propose a transported Muon method that can leverage higher-order smoothness and prove an improved bound of $O(\min\{m, n\}\epsilon^{-\frac{5p-3}{2p-2}})$ when the norm is spectral and the Hessian is Lipschitz. We incorporate practical heuristics into our transported method and use it to train CNNs and transformers, showing the flexibility and compatibility of the transportation technique with neural network training.
\end{enumerate}

\paragraph{Related work.} Our work is most closely related to the literature on neural network optimization methods and optimization under heavy-tailed noise. Due to space limitations, we defer our comments on other relevant topics to Appendix~\ref{app:additional}. Earlier matrix optimizers exploit layerwise matrix structure through spectral or Kronecker-factored preconditioning~\citep{Carlson-2015-Preconditioned, Martens-2015-Optimizing, Grosse-2016-Kronecker, Gupta-2018-Shampoo, Goldfarb-2020-Practical, Ren-2021-Tensor, Duvvuri-2024-Combining, Zhao-2024-Galore, Morwani-2025-Perspective, Vyas-2025-Soap, Yuan-2025-Mars, An-2025-Asgo}. A recent line of work designs optimizers based on the spectral norm: Muon updates the weights using the matrix sign of layerwise gradients or momentum, typically via Newton-Schulz iterations, while Scion casts the updates as linear minimization oracles over spectral-norm balls~\citep{Jordan-2024-Muon, Pethick-2025-Training}. This viewpoint has led to a growing family of matrix optimizers~\citep{Liu-2025-Muon, Li-2025-Normuon, Riabinin-2025-Gluon, Ahn-2025-Dion, Ahn-2025-Dion2, Lau-2025-Polargrad, He-2025-Low, Huang-2025-Limuon, Page-2025-Muonall, Xu-2026-Fismo, Gu-2026-Mano, Gong-2026-Aro, Zhang-2026-Adam, Du-2026-Newton, Li-2026-Intrinsic, Shumaylov-2026-Muon} and faster matrix-sign routines~\citep{Amsel-2026-Polar, Zhang-2026-Gram}.  Despite limited existing work analyzing Muon using spectral-norm smoothness and dual-norm stationarity~\citep{Li-2025-Note, Riabinin-2025-Gluon,Hubler-2026-Free}, it remains unclear whether Muon's convergence rate is optimal in spectral-norm geometry. 

Empirical studies have found heavy-tailed stochastic gradient noise in training neural networks and language models~\citep{Simsekli-2019-Tail, Zhang-2020-Adaptive, Gurbuzbalaban-2021-Heavy, Kunstner-2024-Heavy, Kunstner-2025-Scaling}, motivating updates that reduce sensitivity to raw gradient magnitudes.~\citet{Liu-2025-Nonconvex} show that normalized SGD with momentum attains the optimal heavy-tailed nonconvex rate without clipping. When applied to matrix-valued parameters, their approach gives Frobenius normalization, which is different from Muon. More recently, \citet{Yu-2026-Sign} and \citet{Choudhury-2026-Muon} establish upper bounds for Muon under heavy-tailed noise, assuming optimizer-specific and Frobenius-norm noise and smoothness models, respectively. In concurrent work, \citet{Hubler-2026-Free} prove an upper bound for Muon with $p^\textnormal{th}$-moment noise in the nuclear norm, together with a matching lower bound for zero-respecting algorithms when $\max\{m,n\}/\min\{m,n\}$ is sufficiently large. In contrast, our lower bound imposes standard dimension conditions ($\max\{m,n\}=\widetilde\Omega(\epsilon^{-2})$) and holds for any randomized algorithm with only gradient access, a class that includes Muon and covers the square matrices that are common in practice. We complement it with a matching upper bound for Muon with inexactness under general-norm smoothness and a $p^\textnormal{th}$-moment noise condition in the dual norm, and with an accelerated rate for a transported Muon variant under Hessian smoothness.

\section{Preliminaries and Technical Background}\label{sec:prelim}
We provide an overview of Muon and Scion methods and their norm geometry in the context of neural network optimization. We then present the formal definitions of the function classes and heavy-tailed noise models considered in this paper. 
\subsection{Muon methods and norm geometry}
Throughout this paper, we equip $\mathbb{R}^{m\times n}$ with a general norm $\|\cdot\|$ and denote its dual norm by $\|\cdot\|_\star$. We assume $F^\star := \inf_{X \in \mathbb{R}^{m\times n}} F(X) > -\infty$ and consider
\begin{equation*}
\min_{X \in \mathbb{R}^{m\times n}} F(X). 
\end{equation*}
The key component of Muon and Scion methods is a linear minimization oracle ($\operatorname{LMO}$) over the unit norm ball, defined as
\begin{equation*}
\operatorname{lmo}(S) \in \operatorname{argmin}_{\|X\|\le 1}\langle S,X\rangle .
\end{equation*}
By definition, we have $\|\operatorname{lmo}(S)\| \leq 1$ and $ \langle S,\operatorname{lmo}(S)\rangle=-\|S\|_\star$. An update $X^+=X+\eta\operatorname{lmo}(S)$ using $\operatorname{lmo}(S)$ as the direction therefore fixes the step length in the chosen geometry.

In practice, Scion methods choose the norm ball layerwise using input-output matrix norms for neural network optimization~\citep{Pethick-2025-Training}. In particular, for a weight matrix $W \in \mathbb{R}^{d_{\mathrm{out}}\times d_{\mathrm{in}}}$ and vector norms $\|\cdot\|_\alpha,\|\cdot\|_\beta$, we define 
\begin{equation*}
\|W\|_{\alpha \to \beta} := \sup_{\|z\|_\alpha\le 1}\|Wz\|_\beta .
\end{equation*}
When the layer input is bounded in $\|\cdot\|_\alpha$, the $\|\cdot\|_\beta$ norm of the output can be bounded via $\|W\|_{\alpha\to\beta}$. This input-output interpretation explains why general matrix norms are used in neural network optimizers. Frobenius normalization views $W$ as a vector in Euclidean space. However, it does not measure the layer map $z \mapsto Wz$ through an input-output matrix norm of the form $\|\cdot\|_{\alpha \to \beta}$.

For hidden layers, a common choice of input-output norm is the RMS-to-RMS norm, where $\|z\|_{\mathrm{RMS}}=\frac{1}{\sqrt{d}}\|z\|_2$ for $z \in \mathbb{R}^d$. The corresponding matrix norm is
\begin{equation*}
\|W\|_{\mathrm{RMS}\to\mathrm{RMS}} = \sqrt{\tfrac{d_{\mathrm{in}}}{d_{\mathrm{out}}}}\,
\|W\|_{\mathrm{op}}.
\end{equation*}
Let $S = U\Sigma V^\top$ be an SVD of $S\in \mathbb{R}^{d_{\mathrm{out}}\times d_{\mathrm{in}}}$. We have
\begin{equation*}
\operatorname{lmo}_{\mathrm{RMS}\to\mathrm{RMS}}(S) = -\sqrt{\tfrac{d_{\mathrm{out}}}{d_{\mathrm{in}}}}UV^\top,
\end{equation*}
which is the scaled matrix-sign update used by Muon and its variants~\citep{Jordan-2024-Muon, Bernstein-2024-Old, Liu-2025-Muon, Li-2025-Normuon}. Other normalized updates in neural network optimization can be recovered by different choices of the input-output norm: the $1 \to \mathrm{RMS}$ norm gives column-normalized updates, the $\mathrm{RMS} \to \infty$ norm gives row-normalized updates, and the $1 \to \infty$ norm gives sign updates.

For a feed-forward network with weights and biases $\Theta=(W_1,b_1,\ldots,W_L,b_L)$, we define the norm of the network parameter $\Theta$ as follows:
\begin{equation*}
\|\Theta\| := \max_{\ell\in[L]} \ \tfrac{1}{\rho_\ell}\max\{\|W_\ell\|_{\alpha_\ell \to \beta_\ell}, \|b_\ell\|_{\beta_\ell}\},
\end{equation*}
where $\rho_\ell>0$ is the layerwise radius. 
The LMO over the unit ball in the product space can be decomposed across layers, and hence each layer receives a radius-scaled normalized update in its own input-output geometry. For example, in the RMS-to-RMS case, the update has spectral norm of order $\sqrt{d_{\mathrm{out}}/d_{\mathrm{in}}}$, matching the scaling used for width-stable feature learning and hyperparameter transfer. Our theory keeps the norm abstract but focuses on one weight matrix for simplicity.

\subsection{Assumptions and basic properties}
We present definitions for generalized smooth functions and a $p^\textnormal{th}$-moment heavy-tailed noise model. 
\begin{assumption}\label{assumption:smooth}
There exist $L_0,L_1 \geq 0$ such that, for any $X,Y \in \mathbb{R}^{m\times n}$ satisfying $\|X-Y\| \leq \frac{1}{L_1}$, we have $\|\nabla F(Y)-\nabla F(X)\|_\star \leq (L_0+L_1\|\nabla F(X)\|_\star)\|Y-X\|$. Here, $\frac{1}{L_1}=+\infty$ if $L_1=0$. 
\end{assumption}
Assumption~\ref{assumption:smooth} recovers $L_0$-smoothness when $L_1=0$, while allowing the local smoothness scale to grow with the dual gradient norm $\|\nabla F(X)\|_\star$. This lets our analysis cover objectives with relaxed or unbounded smoothness while still retaining the descent inequality needed for normalized updates.

\begin{assumption}\label{assumption:noise}
There exists an oracle $G: \mathbb{R}^{m\times n} \times \Xi \to \mathbb{R}^{m\times n}$ such that, for every $X \in \mathbb{R}^{m\times n}$, we have $\mathbb{E}[G(X,\xi) \mid X] = \nabla F(X)$ and $\mathbb{E}[\|G(X,\xi)-\nabla F(X)\|_\star^p \mid X] \leq \sigma_0^p+\sigma_1^p\|\nabla F(X)\|_\star^p$ for some constants $\sigma_0,\sigma_1 \geq 0$ and some order $p \in (1,2]$. 
\end{assumption}
Assumption~\ref{assumption:noise} reduces to the classical finite $p^\textnormal{th}$-moment heavy-tailed noise model when $\sigma_1=0$, and the additional $\sigma_1$ term allows the noise scale to grow with the local gradient. In particular, for $p<2$ we do not assume bounded variance or bounded stochastic gradients. Given a query point $X$ and batch size $B$, we write
\begin{equation*}
\bar{G}_B(X):=\tfrac{1}{B}\textstyle\sum_{i=1}^B G(X,\xi^i),
\end{equation*}
where $\xi^1,\ldots,\xi^B$ are \emph{i.i.d.} samples. Thus, when an algorithm chooses $X_t$ from the past and then draws a mini-batch, Assumption~\ref{assumption:noise} implies the conditional unbiasedness, conditional independence, and conditional $p^\textnormal{th}$-moment bounds used in our analysis.

The key difference between our analysis and that of \citet{Liu-2025-Nonconvex} is that the momentum error is controlled in a general dual norm. To highlight this norm-dependent effect, we define
\begin{equation}\label{def:factor}
\tau(\|\cdot\|_\star,m,n,p):=\sup_{\{Z_t\}_{t=1}^T}\tfrac{\mathbb{E}\|\sum_{t=1}^T Z_t\|_\star}{
\mathbb{E}(\sum_{t=1}^T \|Z_t\|_\star^p)^{1/p}},
\end{equation}
where the supremum is over all $T \in \mathbb{N}$ and all integrable $\mathbb{R}^{m\times n}$-valued martingale difference sequences $\{Z_t\}_{t=1}^T$ with respect to their natural filtrations. The ratio is $0$ when the denominator is $0$. The following lemma records the dimension dependence of this martingale factor.
\begin{lemma} 
\label{lemma:dimension-upper-bound}
For any norm $\|\cdot\|_\star$ on $\mathbb{R}^{m\times n}$ and any $p \in (1,2]$, $\tau(\|\cdot\|_\star,m,n,p)$ is finite. In general, this factor can depend on $(m,n,p)$, e.g., $\tau(\|\cdot\|_\textnormal{nuc},m,n,p)=\Theta(\min\{m,n\}^{1-1/p})$.
\end{lemma}
Lemma~\ref{lemma:dimension-upper-bound} identifies the theoretical difference between the Frobenius norm and general norms used by Muon and Scion. For the Frobenius norm, the factor $\tau$ is dimension-free, matching the setting in which a matrix parameter is treated as a vector~\citep{Liu-2025-Nonconvex}. For other matrix norms, the estimator error is measured in the corresponding dual norm, and this can introduce dimension dependence. For example, spectral-norm LMO updates require controlling the stochastic error in the nuclear norm, where $\tau(\|\cdot\|_\textnormal{nuc},m,n,p)=\Theta(\min\{m,n\}^{1-1/p})$. Thus, the dimension dependence studied in this paper is a consequence of combining heavy-tailed noise with general input-output matrix norm geometries used by neural network optimizers. The proof of Lemma~\ref{lemma:dimension-upper-bound} is deferred to Appendix~\ref{app:proofs}.
\begin{assumption}\label{assumption:hesslip}
There exists $L_2 \geq 0$ such that, for any $X,Y \in \mathbb{R}^{m\times n}$ satisfying $\|Y-X\|\leq \frac{1}{L_1}$, we have $\|\nabla F(Y)-\nabla F(X)-\nabla^2F(X)[Y-X]\|_{\star} \le L_2\|Y-X\|^2$. Here, $\frac{1}{L_1}=+\infty$ if $L_1=0$.
\end{assumption}
Assumption~\ref{assumption:hesslip} is implied by the standard Hessian Lipschitzness condition in Euclidean geometry. This additional higher-order smoothness assumption allows us to design accelerated algorithms.

In practice, Muon and Scion usually use an inexact LMO to accelerate computation on GPUs, e.g., Newton-Schulz~\citep{Jordan-2024-Muon} or PolarExpress \citep{Amsel-2026-Polar}.
We model such inexactness in general norms by requiring a $\delta$-approximation $\operatorname{lmo}_\delta:\mathbb{R}^{m\times n}\to\mathbb{R}^{m\times n}$ to satisfy the following condition.
\begin{assumption}
\label{assumption:inexactness}
There exists $\delta\in[0,1)$ such that, for all $X\in\mathbb{R}^{m\times n}$,
\begin{equation*}
\|\operatorname{lmo}_\delta(X)\|\leq 1,
\quad
\langle X,\operatorname{lmo}_\delta(X)\rangle \leq -(1-\delta)\|X\|_\star, \quad \mathrm{and} \quad
\operatorname{lmo}_\delta(0) = 0.
\end{equation*}
\end{assumption}
Setting $\delta=0$ recovers an exact LMO. Lemma~\ref{lem:polynomial-inexact-lmo} shows that Newton--Schulz (Algorithm~\ref{alg:ns}) and PolarExpress (Algorithm~\ref{alg:pe}) satisfy
Assumption~\ref{assumption:inexactness}.
The proof and the schemes are given in Appendix~\ref{app:proofs}.
\begin{lemma}
\label{lem:polynomial-inexact-lmo}
Let $K\geq1$, $\|\cdot\|=\|\cdot\|_\textnormal{op}$, and $\|\cdot\|_\star=\|\cdot\|_\textnormal{nuc}$. The corresponding inexact LMO of $K$-step Newton-Schulz or $K$-step PolarExpress satisfies Assumption~\ref{assumption:inexactness} with $\delta_\textnormal{NS}$ or $\delta_\textnormal{PE}$, where
\begin{align}
\label{eq:inexactness-lemma}
\delta_\textnormal{NS}
\leq \min\{m,n\}\left(\tfrac{8}{15}\right)^K,\quad
\delta_\textnormal{PE}
\leq \min\{m,n\} \sqrt{K}\,3^{-K/2}+2\exp(-K).
\end{align}
\end{lemma}
Lemma~\ref{lem:polynomial-inexact-lmo} shows that the inexactness parameters of Newton-Schulz and PolarExpress can be made arbitrarily small by increasing $K$. In particular, both $\delta_\textnormal{NS}$ and $\delta_\textnormal{PE}$ are bounded by a numerical constant $<1$ after $O(\log \min\{m,n\})$ inner iterations. 

\subsection{Oracle complexity of smooth optimization in operator-norm geometry}
We study the dimension dependence that is unavoidable for gradient-based methods when the problem class is defined using an input-output norm and the stochastic oracle has heavy-tailed noise. We focus on the operator norm for simplicity. We use the black-box first-order oracle model~\citep{Arjevani-2023-Lower} to measure the number of oracle calls needed to optimize in the operator-norm geometry, where the unknown objective $f$ belongs to a known class $\FCal$ of gradient Lipschitz functions on $\br^{m \times n}$, and accuracy is measured by the expected gradient nuclear norm.

Let $\Delta,L>0$. Define $\mathcal{F}_{\mathrm{op}}(m,n,\Delta,L)$ as the class of differentiable functions $F:\mathbb{R}^{m\times n} \to \mathbb{R}$ satisfying
\begin{equation*}
F(0)-\inf_{X \in \mathbb{R}^{m\times n}} F(X) \leq \Delta, \quad \|\nabla F(X)-\nabla F(Y)\|_{\mathrm{nuc}} \leq L\|X-Y\|_{\mathrm{op}} \textnormal{ for all } X,Y \in \mathbb{R}^{m\times n}.
\end{equation*}
We use the stochastic first-order oracle model with only stochastic gradient access. Each oracle $\mathsf{O}$ consists of a distribution $P_\xi$ on a measurable space $\Xi$ and a mapping $\mathsf{O}_F(X,\xi)=g(X,\xi)$ such that, for every $F \in \mathcal{F}_{\mathrm{op}}(m,n,\Delta,L)$, we have
\begin{equation*}
\mathbb{E}[g(X,\xi) \mid X] = \nabla F(X), \quad \mathbb{E}[\|g(X,\xi)-\nabla F(X)\|_{\mathrm{nuc}}^p \mid X] \leq \sigma_0^p.
\end{equation*}
We denote the set of all such oracles by $\mathcal{O}_p(\sigma_0)$.

We use the randomized algorithm model. Let $\mathcal{A}_\textnormal{rand}$ denote the class of randomized first-order algorithms using the oracle $\mathsf{O}_F$. At round $t$, an algorithm $A[\mathsf{O}_F] \in \mathcal{A}_\textnormal{rand}$ chooses a query point $X_{A[\mathsf{O}_F]}^{(t)}$ according to a decision rule measurable with respect to its internal random seed and all previous oracle observations. The oracle then draws $\xi^{(t)} \sim P_\xi$ independently and returns $g(X_{A[\mathsf{O}_F]}^{(t)},\xi^{(t)})$. 

Let $\mathcal{P}(\mathcal{F}_{\mathrm{op}}(m,n,\Delta,L))$ be the set of probability measures over $\mathcal F_{\mathrm{op}}(m,n,\Delta,L)$. For any $\epsilon>0$ and $p \in (1,2]$, we define the worst-case number of oracle rounds needed to query an expected $\epsilon$-stationary point by
\begin{equation}
\label{eq:def-complexity}
\begin{aligned}
&\mathfrak{m}^{\mathrm{rand}}_{\epsilon,p}(m,n,\Delta,L,\sigma_0):=
\\
&\quad\sup_{\mathsf{O} \in \mathcal{O}_p(\sigma_0)} \sup_{P_F \in \mathcal{P}(\mathcal{F}_{\mathrm{op}}(m,n,\Delta,L))} \inf_{A \in \mathcal{A}_{\mathrm{rand}}} \inf\left\{N \mid \mathbb{E}_{F\sim P_F, A[\mathsf{O}_F]} \|\nabla F(X_{A[\mathsf{O}_F]}^{(N)})\|_{\mathrm{nuc}} \le\epsilon \right\}.
\end{aligned}  
\end{equation}
The expectation is taken over the random problem instance, the internal randomness of the algorithm, and the stochastic oracle.

Our oracle model differs from that of~\citep{Arjevani-2023-Lower}, whose oracle provides the exact function value as well as a stochastic gradient. We adopt this gradient-only oracle model because it reflects how optimizers are used in practice. In neural network training, the mini-batch loss is almost always available and the full-batch loss is almost always unavailable. However, widely used optimizers such as Adam, Muon, and NorMuon~\citep{Li-2025-Normuon} do not use any loss value, and their updates depend on the observed stochastic gradients only. There are exceptions. \citet{Crawshaw-2026-an} combine Muon with Polyak step size, which requires the mini-batch loss at every iteration. More generally, some practical heuristics rely on the internal states of the network. For example, MuonClip~\citep{Team-2025-Kimi} applies QK-Clip, which rescales the query and key projection weights. Neither our oracle model nor the standard one can describe such methods. The corresponding lower bound in the standard oracle model remains open.

\section{Main Results}\label{sec:results}
We establish a dimension-dependent lower bound for any stochastic first-order method using our oracle model under spectral-norm geometry and heavy-tailed noise. We then show that batched Muon achieves the same rate. Finally, we introduce a transported Muon method and prove an improved rate under Hessian Lipschitzness. The lower bound is stated for the spectral norm geometry, where stationarity is measured in the nuclear norm. The upper bounds are stated for a general norm $\|\cdot\|$ on $\mathbb{R}^{m\times n}$ with dual norm $\|\cdot\|_\star$. Throughout this section, when stating upper bounds, we write $\tau_\star := \tau(\|\cdot\|_\star,m,n,p)$ where $\tau(\cdot)$ is the martingale factor (cf. Eq.~\eqref{def:factor}). For spectral norm LMO updates, we have $\|\cdot\|=\|\cdot\|_{\mathrm{op}}$ and $\|\cdot\|_\star=\|\cdot\|_{\mathrm{nuc}}$. In this case, we have $\tau_\star=\Theta(\min\{m,n\}^{1-1/p})$. All proofs are deferred to Appendix~\ref{app:proofs}.

\subsection{Dimension-dependent lower bound}
\label{subsec:dimension-dependent-lower-bound}
\begin{theorem}\label{thm:main}
For any $p\in(1,2]$, there exist constants $c_p,c'_p,C_p>0$, depending only on
$p$, such that the following holds. For any $m,n\ge 1$ and
$\Delta,L,\sigma_0>0$, if
\begin{equation}
\label{eq:dim-lb}
0<\epsilon\le c'_p\min\{\sqrt{\Delta L},\sigma_0\},\quad \max\{m,n\} \geq C_p \tfrac{\Delta L}{\epsilon^2}(\log(3+\tfrac{\Delta L}{\epsilon^2}) + \tfrac{p}{p-1}\log(3+\tfrac{\sigma_0}{\epsilon})),
\end{equation}
we have (cf. Eq.~\eqref{eq:def-complexity})
\begin{equation*}
\mathfrak{m}^{\mathrm{rand}}_{\epsilon,p}(m,n,\Delta,L,\sigma_0) \geq c_p\min\{m,n\}\Delta L \sigma_0^{\frac{p}{p-1}}\epsilon^{-\frac{3p-2}{p-1}}.
\end{equation*}
\end{theorem}
The key is the multiplicative factor $\min\{m,n\}$, which comes from the interaction between spectral-norm geometry and nuclear-norm stochastic gradient estimation. \citet{Hubler-2026-Free} concurrently prove the same lower bound for zero-respecting algorithms, assuming $\max\{m,n\} / \min\{m,n\}\geq \Delta L/\epsilon^2$. We remove this condition, so that our bound covers any height/width ratio, by using techniques from \citet{Jin-2026-Tight}, in which the classical hard instance of \citet{Arjevani-2023-Lower} is adapted to explicitly encode an i.i.d.\ uniform $\pm 1$ tensor. We restrict algorithms to access only stochastic gradients because function value queries could leak information about this tensor and break the independence. Neither of the two lower bounds proved by \citet{Hubler-2026-Free} simultaneously covers square matrices and achieves the tight accuracy dependence. Our condition on $\max\{m,n\}$ is inherited from \citet{Arjevani-2023-Lower}.

\paragraph{Proof Sketch.} We assume $m\leq n$ without loss of generality and prove by (1) constructing hard functions (2) showing limited progress. We do (1). Let $f_K:\mathbb{R}^K\to\mathbb{R}$ be the hard function of \citet{Arjevani-2023-Lower}, whose gradient stays large until all $K$ coordinates have absolute value $>\frac12$ and reveals at most one new coordinate at a time. Let $N=\lfloor n/\max\{m,K\}\rfloor$ and $c_{i,j}^a:=N((i+j-2)\bmod \max\{m,K\})+a$. We assign to each $j\in[K]$ a set of $mN$ distinct entries $\{(i,c_{i,j}^a)\}_{i\in[m],a\in[N]}$ of $X$, with $N$ entries in each row, such that different rows use disjoint columns. Given i.i.d.\ fair signs $\theta\in\{\pm1\}^{m\times K\times N}$, $y_{i,j}^a:=\sqrt{N}X_{i,c_{i,j}^a}/\rho$, $A=\Theta(\epsilon^2/L)$, $K=\Theta(\Delta L/\epsilon^2)$, and $\rho=A/4\epsilon$, we follow \citet{Jin-2026-Tight} and define the hard function as $F_{\theta,K}(X)$, where for $k=0,\ldots,K$,
\begin{align*}
F_{\theta,k}(X):=Af_K(q_1^\theta(X),\ldots,q_k^\theta(X),0,\ldots,0)+R(X),\
R(X):=\tfrac{8A}{mN}\textstyle\sum_{r,i,a=1}^{K,m,N}\log\cosh y_{i,r}^a,
\end{align*}
and $q_j^\theta(X):=\tfrac{4}{mN}\textstyle\sum_{i=1}^m\sum_{a=1}^N\theta_{i,j}^a\tanh y_{i,j}^a$. Lemma~\ref{lemma:stochastic-matrix} shows that for all $X$, $|q_K^\theta(X)|\leq\frac12$ implies $\|\nabla F_{\theta,K}({X})\|_{\mathrm{nuc}}>2\epsilon$. It suffices to show that, for any algorithm, with probability at most $\frac14$ some query among its first $T+1$ query points $X_1,\ldots,X_{T+1}$ satisfies $|q_K^\theta(X_t)|>\tfrac12$.
 
We define the stochastic gradient oracle. Let $k=\operatorname{prog}(X):=\max(\{j\in[K]:|q_j^\theta(X)|>\frac12\}\cup\{0\})$. If $k=K$, the oracle returns $\nabla F_{\theta,K}(X)$. Otherwise, we write $H_k^\theta(X):=\nabla F_{\theta,k+1}(X)-\nabla F_{\theta,k}(X)$, where $H_k^\theta(X)$ is supported on $(i,c_{i,j}^a)_{i\leq m,a\leq N}$ where $j\in J_k:=\{k,k+1\}\setminus\{0\}$. Let $\nu:=(2000\epsilon/\sigma_0)^{\frac{p}{p-1}}$ \citep{Liu-2025-Nonconvex}, and draw independent $\chi\sim\mathrm{Bernoulli}(\nu)$, $(I,\ell)$ uniform on $[m]\times J_k$, and $g\sim\mathcal{N}(0,I_N)$. We define
\begin{equation*}
G_\theta(X,\xi):=\nabla F_{\theta,k}(X)+\tfrac{\chi}{\nu}\,\mathcal{E}_{I,\ell}\big(m|J_k|h_k+\tfrac{800\epsilon}{\sqrt N}g\big),\quad h_k:=\big([H_k^\theta(X)]_{I,c_{I,\ell}^a}\big)_{a=1}^N,
\end{equation*}
where $\mathcal{E}_{I,\ell}(v)$ places $v_a$ at entry $(I,c_{I,\ell}^a)$ for $a=1,\ldots,N$. That is, with probability $\nu$, the oracle reveals a rescaled, Gaussian-masked copy of one random row of $H_k^\theta(X)$. 
 
To do (2), for any $k<K$, we define an oracle for comparison that depends only on $\theta_{\leq k}$. Let $Q_k^{\theta_{\leq k}}(X,\xi):=\nabla F_{\theta,k}(X)+\tfrac{\chi}{\nu}\mathcal{E}_{I,\ell}(\tfrac{800\epsilon}{\sqrt{N}}g)$ without the $h_k$ signal. If $\operatorname{prog}(X)=k<K$, $G_\theta\not=Q_k^{\theta_{\leq k}}$ only if $\chi=1$. For $j\leq K$, let $\widetilde{G}_j$ equal $Q_{j-1}^{\theta_{\leq j-1}}$ if $\operatorname{prog}(X)=j-1$ and $G_\theta$ otherwise. We compare the algorithm runs under $G_\theta$ and $\widetilde{G}_j$. Let $E_j$ and $\widetilde E_j$ be the events that $\operatorname{prog}(X_t)\geq j$ for some $t\leq T+1$ under $G_\theta$ and $\widetilde{G}_j$, respectively, and let $L_j:=\sum_{t=1}^T\mathbf{1}\{\operatorname{prog}(X_t)=j-1\}$ under $G_\theta$. If $\mathbb{P}(E_j)=O(\nu \mathbb{E}L_j /m+1/(mN))$, since $\sum_{j=1}^KL_j\leq T$ and $E_K\subseteq E_j$ for all $j$, averaging gives $\mathbb{P}(E_K)=O(\nu T/(mK)+1/(mN))$, which is at most $\frac14$ if $T=\Theta(Km/\nu)=\Theta(m\Delta L\sigma_0^{\frac{p}{p-1}}\epsilon^{-\frac{3p-2}{p-1}})$. Thus, it suffices to bound $\mathbb{P}(E_j)$. Under $\widetilde{G}_j$, the queries are determined by $\theta_{\leq j-1}$ until $\operatorname{prog}(X_t)\geq j$, since $\widetilde{G}_j$ depends only on $\theta_{\leq j-1}$ if $\operatorname{prog}(X_t)<j$. If $X$ is independent of $\theta_j$, Hoeffding's inequality gives $\mathbb{P}(|q_j^\theta(X)|>\tfrac12|X)\leq2e^{-mN/128}$. Hence, a union bound gives $\mathbb{P}(\widetilde E_j)\leq2K(T+1)e^{-mN/128}$. $\mathbb{P}(E_j)$ is then bounded by KL divergence. Since $\widetilde{G}_j$ differs from $G_\theta$ only at queries with progress $j-1$ and $\operatorname{KL}(G_\theta(X,\cdot)\|Q_k^{\theta_{\leq k}}(X,\cdot))\leq\tfrac{N\nu}{2}$ (proved by Gaussianity in Lemma~\ref{lemma:stochastic-oracle}), the KL divergence between the underlying distributions of the runs under $G_\theta$ and $\widetilde{G}_j$ is at most $\frac{N\nu}{2}\mathbb{E}L_j$. This allows us to bound $\mathbb{P}(E_j)$ by the bound on $\mathbb{P}(\widetilde{E}_j)$.

\subsection{Smooth and nonconvex problems}\label{sec:smooth-and-nonconvex-problems}
\begin{algorithm2e}[!t]
\caption{Muon with heavy-ball momentum}\label{algorithm:Muon-HB}
\LinesNumbered
\SetAlgoVlined
\SetKwInput{KwInit}{Initialization}
\KwIn{$T \geq 1$, $\delta\in[0,1)$, $\beta_t \in [0,1)$, $\eta_t > 0$ for $0 \leq t \leq T-1$, and batch size $B$.}
\KwInit{$X_0$, $\bar{G}_0 := \frac{1}{B} \sum_{i=1}^{B} G(X_0,\xi_0^i)$, $m_1:=\bar{G}_0$, $X_1=X_0+\eta_0\operatorname{lmo}_\delta(m_1)$.}
\For{$t = 1, \ldots, T-1$}{
  $\bar{G}_t = \frac{1}{B} \sum_{i=1}^{B} G(X_t,\xi_t^i)$\;
  $m_{t+1}=\beta_t m_{t} + (1-\beta_t)\bar{G}_t$\;
  $X_{t+1} = X_t + \eta_t \operatorname{lmo}_\delta(m_{t+1})$\;
}
\KwOut{$\widetilde{X}_T$ is uniformly chosen from $\{X_0, \ldots, X_{T-1}\}$.}
\end{algorithm2e}
We move to the upper bound. Algorithm~\ref{algorithm:Muon-HB} is a batched momentum variant of Muon or Scion. The LMO fixes the update scale in the primal norm, while the momentum estimator is controlled in the dual norm. This dual-norm control is precisely where the martingale factor $\tau_\star$ enters.

\begin{theorem}\label{thm:nonconvex-smooth}
Suppose that Assumptions~\ref{assumption:smooth}, \ref{assumption:noise}, and~\ref{assumption:inexactness} hold for some $p\in(1,2]$. Let $\sigma_0,L_0>0$, and $\Delta_0:=F(X_0)-F^\star>0$. Then, for any sufficiently small $\epsilon>0$, there exists an integer $T>0$ such that Algorithm~\ref{algorithm:Muon-HB} with $B=\max\{1,\lceil (\tfrac{16\tau_\star\sigma_1}{1-\delta})^{\frac{p}{p-1}}\rceil\}$, $\alpha=\min\{1, B(\tfrac{(1-\delta)\epsilon}{120\tau_\star\sigma_0})^{\frac{p}{p-1}}\}$, $\eta_t\equiv\eta=\min\{\sqrt{\tfrac{\alpha\Delta_0}{L_0T}}, \tfrac{(1-\delta)\alpha}{8L_1}\}$, and $\beta_t\equiv\beta =1-\alpha$ satisfies $\mathbb E\|\nabla F(\widetilde X_T)\|_\star\le\epsilon$ and the total number of calls of the stochastic gradient oracle is bounded by
\begin{equation*}
O(\tau_\star^{\frac{p}{p-1}} \sigma_0^{\frac{p}{p-1}} L_0\Delta_0(1-\delta)^{-\frac{3p-2}{p-1}}\epsilon^{-\frac{3p-2}{p-1}}).
\end{equation*}
If $\|\cdot\|=\|\cdot\|_\textnormal{op}$ and $\|\cdot\|_\star=\|\cdot\|_\textnormal{nuc}$, Algorithm~\ref{algorithm:Muon-HB} with the parameters above satisfies $\mathbb E\|\nabla F(\widetilde X_T)\|_\textnormal{nuc}\le\epsilon$ and the total number of calls of the stochastic gradient oracle is bounded by
\begin{align*}
O(\min\{m,n\} \sigma_0^{\frac{p}{p-1}} L_0\Delta_0(1-\delta)^{-\frac{3p-2}{p-1}}\epsilon^{-\frac{3p-2}{p-1}}).
\end{align*}
\end{theorem}
Under spectral-norm geometry, Theorem~\ref{thm:nonconvex-smooth} matches Theorem~\ref{thm:main} up to constants. The dimension dependence is the cost of estimating heavy-tailed gradient noise in the nuclear norm. 

\subsection{Highly smooth and nonconvex problems}
To accelerate Algorithm~\ref{algorithm:Muon-HB}, we introduce a higher-order smoothness condition (Assumption~\ref{assumption:hesslip}). Under this assumption, we propose Algorithm~\ref{algorithm:transported-Muon}. Inspired by gradient transportation \citep{Cutkosky-2020-Momentum}, this algorithm uses an auxiliary sequence $Y_t$ for gradient evaluation while maintaining the primary sequence $X_t$ for gradient-descent-style updates. By evaluating the gradient at $Y_t$, certain first-order drift terms are replaced by more controllable Hessian curvature error terms. Our algorithm resembles AMUSE~\citep{Kim-2026-Amuse}. However, our analysis relies on drift cancellation under Hessian smoothness while AMUSE is schedule-free and fits the river valley landscape.
\begin{algorithm2e}[!t]
\caption{Transported Muon}\label{algorithm:transported-Muon}
\LinesNumbered
\SetAlgoVlined
\SetKwInput{KwInit}{Initialization}
\KwIn{$T\geq 1$, $\delta\in[0,1)$, $\beta_t\in[0,1)$, $\eta_t>0$, and batch size $B$.}
\KwInit{$X_0=Y_0$, $\bar G_0 := \frac{1}{B} \sum_{i=1}^{B} G(X_0,\xi_0^i)$, $m_1:=\bar{G}_0$, $X_1 := X_0 + \eta_0\operatorname{lmo}_\delta(m_1)$.}
\For{$t = 1, \ldots, T-1$}{
  $Y_t = X_t + \frac{\beta_t}{1-\beta_t}(X_t-X_{t-1})$\;
  $\bar G_t = \frac{1}{B} \sum_{i=1}^B G(Y_t,\xi_t^i)$\;
  $m_{t+1} = \beta_t m_t + (1-\beta_t) \bar{G}_t$\;
  $X_{t+1} = X_t + \eta_t \operatorname{lmo}_\delta(m_{t+1})$\;
}
\KwOut{$\widetilde X_T$ is uniformly chosen from $\{X_0,\ldots,X_{T-1}\}$.}
\end{algorithm2e}

\begin{theorem}\label{thm:transported-uSCG}
Suppose that Assumptions~\ref{assumption:smooth}, \ref{assumption:noise}, \ref{assumption:hesslip}, and \ref{assumption:inexactness} hold for some $p\in(1,2]$. Let $\sigma_0,L_0,L_2>0$, and $\Delta_0:=F(X_0)-F^\star>0$. Then, for any sufficiently small $\epsilon>0$, there exists an integer $T\geq1$ such that Algorithm~\ref{algorithm:transported-Muon} with $B=\lceil(\max\{1,\tfrac{8\tau_\star\sigma_1}{1-\delta},\tfrac{4\tau_\star\sigma_1L_0}{\sqrt{(1-\delta)L_2\epsilon}}\})^{\frac{p}{p-1}}\rceil$, $\alpha=\min\{1,B(\tfrac{(1-\delta)\epsilon}{32\tau_\star\sigma_0})^{\frac{p}{p-1}},\allowbreak\tfrac{\sqrt{(1-\delta)L_2\epsilon}}{L_0}\}$, $\eta_t\equiv\eta=\min\{\tfrac{(1-\delta)\alpha}{8L_1},\tfrac{\alpha}{8}\sqrt{\tfrac{(1-\delta)\epsilon}{L_2}}\}$, and $\beta_t\equiv\beta = 1-\alpha$ satisfies $\mathbb E\|\nabla F(\widetilde X_T)\|_\star\le\epsilon$ and the total number of calls of the stochastic gradient oracle is bounded by
\begin{equation*}
O( \tau_\star^{\frac{p}{p-1}}\sigma_0^{\frac{p}{p-1}} \sqrt{L_2}\Delta_0(1-\delta)^{-\frac{5p-3}{2p-2}}\epsilon^{-\frac{5p-3}{2p-2}}
),\ \textnormal{where}\ \tau_\star^{\frac{p}{p-1}}=\Theta(\min\{m,n\})\ \textnormal{if}\ \|\cdot\|=\|\cdot\|_\textnormal{op}.
\end{equation*}
\end{theorem}
For all sufficiently small $\epsilon>0$, the convergence rate improves the exponent of $\epsilon$ by $1/2$ compared with Theorem~\ref{thm:nonconvex-smooth}, at the cost of the additional smoothness condition in Assumption~\ref{assumption:hesslip}.

\section{Experiments}\label{sec:exp}
We evaluate practical versions of Algorithms~\ref{algorithm:Muon-HB} and \ref{algorithm:transported-Muon} on image classification~\citep{Jordan-2024-Single} and LLM pretraining~\citep{nanochat}. We approximate each LMO with 5 PolarExpress steps~\citep{Amsel-2026-Polar} and apply the normalization of NorMuon \citep{Li-2025-Normuon} to every Muon variant except the Muon baseline. Setups and additional results are deferred to Appendix~\ref{sec:app-exp}. Algorithm~\ref{algorithm:transported-Muon} resembles a Nesterov lookahead step. Following NAdam~\citep{Dozat-2016-Incorporating} and several Muon variants~\citep{Jordan-2024-Muon, Liu-2025-Muon, nanochat}, we use the following practical Nesterov update:
\begin{equation}\label{eq:nesterov-momentum}
\bar{G}_t = \tfrac{1}{B}\sum_{i=1}^B G(X_t,\xi_t^i),\ m_{t+1} = \beta_t m_t+(1-\beta_t)\bar{G}_t,\ 
X_{t+1} = X_t+\eta_t\operatorname{lmo}_\delta\left(\beta_tm_{t+1}+(1-\beta_t)\bar{G}_t\right).
\end{equation}
We also evaluate a transported update based on Algorithm~\ref{algorithm:transported-Muon} with $Y_t=X_t+\gamma(X_t-X_{t-1})$. For CNNs, we use the theoretical choice $\gamma=\beta/(1-\beta)$. For LLMs, we tune $\gamma$ because the theoretical choice was unstable. The resulting update is
\begin{equation}\label{eq:TUSCG-alpha}
\begin{aligned}
Y_t &= X_t+\gamma(X_t-X_{t-1}), & \bar{G}_t &= \tfrac{1}{B}\textstyle\sum_{i=1}^B G(Y_t,\xi_t^i), \\
m_{t+1} &= \beta_t m_t+(1-\beta_t)\bar{G}_t, & X_{t+1} &= X_t+\eta_t\operatorname{lmo}_\delta(m_{t+1}).
\end{aligned}
\end{equation}

\subsection{CNN experiments}
\begin{table}[!t]
\caption{Test loss and accuracy of CIFARNET, reported as mean $\pm$ standard error over 20 seeds.}
\label{tab:cifarnet-optimizer-comparison}
\begin{center}
\begin{tabular}{|l|c|c|c|c|}
\hline
& \multicolumn{2}{c|}{CIFAR-10} & \multicolumn{2}{c|}{CIFAR-100} \\\hline
Optimizer & Test loss & Test acc.  & Test loss & Test acc.  \\\hline
AdamW 
& $0.4297_{\pm 7.7e{-}4}$ & $92.37_{\pm 0.04}\%$
& $1.4783_{\pm 1.3e{-}3}$ & $68.93_{\pm 0.05}\%$ \\\hline
SGD (heavy-ball)
& $0.4285_{\pm 7.5e{-}4}$ & $92.37_{\pm 0.03}\%$
& $1.4968_{\pm 1.8e{-}3}$ & $68.80_{\pm 0.06}\%$ \\\hline
Muon
& $\mathbf{0.3871_{\pm 6.9e{-}4}}$ & $\underline{93.88}_{\pm 0.04}\%$
& $\mathbf{1.4108_{\pm 1.7e{-}3}}$ & $72.12_{\pm 0.06}\%$ \\\hline
NorMuon 
& $0.3906_{\pm 7.3e{-}4}$ & $93.81_{\pm 0.04}\%$
& $1.4264_{\pm 1.3e{-}3}$ & $\underline{72.16}_{\pm 0.05}\%$ \\\hline
NorMuon with Eq.~\eqref{eq:nesterov-momentum}
& $\underline{0.3886}_{\pm 7.8e{-}4}$ & $93.87_{\pm 0.04}\%$
& $1.4153_{\pm 1.2e{-}3}$ & $72.07_{\pm 0.07}\%$ \\\hline
Algorithm~\ref{algorithm:transported-Muon}
& $0.3902_{\pm 6.6e{-}4}$ & $\mathbf{94.09_{\pm 0.03}}\%$
& $\underline{1.4117}_{\pm 2.2e{-}3}$ & $\mathbf{72.31_{\pm 0.07}}\%$ \\\hline
\end{tabular}
\end{center}
\end{table}

We train a CIFARNET model~\citep{Jordan-2024-Single, Kim-2026-Convergence} with 2M parameters on CIFAR-10 and CIFAR-100~\citep{Krizhevsky-2009-Learning}. We apply each compared optimizer to the six convolutional layers of the model, with each layer's weights flattened into a matrix, and train all other parameters with SGD, following standard practice~\citep{Jordan-2024-Muon}. We train for 50 epochs with a batch size of 512 and select hyperparameters by accuracy on a held-out validation set from a grid search. We show the test loss and accuracy of SGD with heavy-ball momentum, AdamW, Muon, NorMuon with heavy-ball or Nesterov momentum, and Algorithm~\ref{algorithm:transported-Muon} in Table~\ref{tab:cifarnet-optimizer-comparison}. The Muon-family methods outperform AdamW and SGD on both metrics. Algorithm~\ref{algorithm:transported-Muon} attains the highest test accuracy, and Muon attains the lowest test loss.

\subsection{LLM experiments}
\begin{table}[!t]
\caption{Validation loss of the nanochat models. Results for the 287M and 539M models are reported as mean $\pm$ standard error over 6 seeds, and results for the 1.39B model use 1 seed.}
\label{tab:nanochat-optimizer-comparison}
\begin{center}
\begin{tabular}{|l|c|c|c|}\hline
& 287M & 539M & 1.39B \\\hline
Optimizer
& Val. Loss
& Val. Loss
& Val. Loss \\\hline
AdamW 
& $2.9822_{\pm 2e{-}3}$ 
& $2.7266_{\pm 3e{-}3}$ 
& $2.4456$\\\hline
Muon 
& $2.8549_{\pm 2e{-}3}$ 
& $2.6283_{\pm 2e{-}3}$ 
& $2.3724$ \\\hline
NorMuon 
& $2.8400_{\pm 1e{-}3}$ 
& $2.6166_{\pm 5e{-}4}$ 
& $2.3526$   \\\hline
NorMuon with Eq.~\eqref{eq:nesterov-momentum}
& $2.8419_{\pm 8e{-}4}$ 
& $2.6167_{\pm 6e{-}4}$
& $2.3543$ \\\hline
NorMuon (scheduled Nesterov momentum)
& $\mathbf{2.8360_{\pm 9e{-}4}}$ 
& $\mathbf{2.6108_{\pm 2e{-}4}}$ 
& $\mathbf{2.3476}$  \\\hline
Algorithm~\ref{algorithm:transported-Muon} with Eq.~\eqref{eq:TUSCG-alpha}
& $\underline{2.8397}_{\pm 1e{-}3}$ 
& $\underline{2.6165}_{\pm 6e{-}4}$ 
& $\underline{2.3525}$  \\\hline
\end{tabular}
\end{center}
\end{table}

We train decoder-only nanochat models~\citep{nanochat} with 287M, 539M, and 1.39B parameters on NVIDIA ClimbMix~\citep{Diao-2025-Nemotron}. All models use a sequence length of 2048 and are trained on 882M, 1.88B, and 5.84B tokens, respectively. We apply each compared optimizer only to the transformer matrices and train all other parameters with AdamW, following standard practice~\citep{Jordan-2024-Muon}. We choose hyperparameters based on a grid search. For Algorithm~\ref{algorithm:transported-Muon} with Eq.~\eqref{eq:TUSCG-alpha}, we choose $\gamma=0.005,0.0075,0.01$ for the 287M, 539M, and 1.39B models, respectively. 

\begin{figure}[t]
\centering
\includegraphics[width=\linewidth]{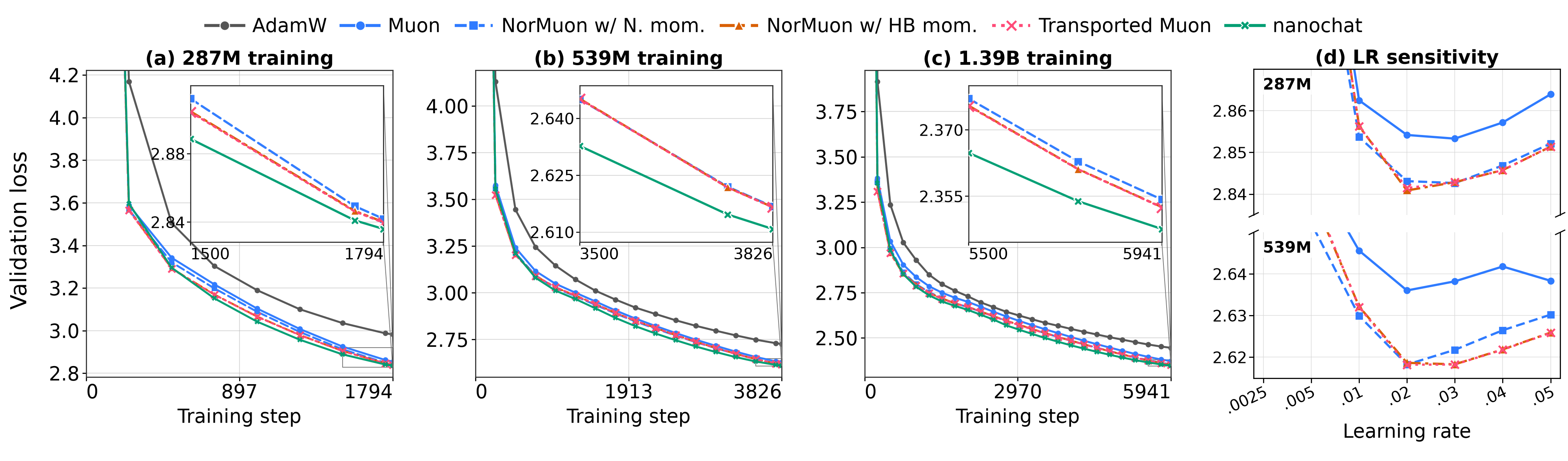}
\caption{Validation loss curves for the (a) 287M, (b) 539M, and (c) 1.39B nanochat models, and (d) learning rate sensitivity 
on the 287M and 539M models.}
\label{fig:nanochat-loss-curves}
\end{figure}

We show the validation loss of AdamW, Muon, NorMuon with heavy-ball, Nesterov, or scheduled Nesterov momentum (the nanochat baseline~\citep{nanochat}), and Algorithm~\ref{algorithm:transported-Muon} with Eq.~\eqref{eq:TUSCG-alpha} in Table~\ref{tab:nanochat-optimizer-comparison} and Figure~\ref{fig:nanochat-loss-curves}(a)--(c). All Muon-family methods attain lower validation loss than AdamW, and NorMuon with scheduled Nesterov momentum attains the lowest loss at every model size. Algorithm~\ref{algorithm:transported-Muon} with Eq.~\eqref{eq:TUSCG-alpha} is compatible with LLM pretraining. In Figure~\ref{fig:nanochat-loss-curves}(d), we show that NorMuon variants and Algorithm~\ref{algorithm:transported-Muon} with Eq.~\eqref{eq:TUSCG-alpha} perform robustly across learning rates.

\section{Conclusion}\label{sec:conclu}
We studied stochastic nonconvex matrix optimization in general-norm geometry and heavy-tailed noise. The motivation is that spectral-norm updates are already used in modern matrix optimizers and are closely connected to hyperparameter transfer, but their theoretical guarantees remain incomplete beyond Frobenius geometry. We showed that, under heavy-tailed noise, spectral-norm geometry affects the difficulty of stochastic optimization and introduces dimension dependence. We established a dimension-dependent lower bound for any method using only stochastic gradients, proved that a batched Muon method achieves the matching upper bound in spectral-norm geometry, and proposed a transported Muon method with an improved rate under Hessian Lipschitzness. Experiments on CNNs and transformer models show that transportation techniques are compatible with practical training pipelines. Future directions include developing principled momentum and transportation techniques for other matrix optimizers~\citep{Martens-2015-Optimizing, Gupta-2018-Shampoo, Vyas-2025-Soap} and applying these ideas to larger-scale LLM pretraining with schedule-free parameter tuning~\citep{Defazio-2024-Road}.

\section*{Acknowledgments}
We sincerely appreciate Buzz High Performance Computing (\hyperlink{https://www.buzzhpc.ai}{\texttt{https://www.buzzhpc.ai}}, \texttt{info@buzzhpc.ai}) for providing computational resources and support for this work.

\bibliographystyle{plainnat}
\bibliography{ref}

\newpage \appendix
\section{Further Related Work}\label{app:additional}
We make some comments on other topics, including more discussions on neural network optimization methods and optimization under heavy-tailed noise, the theoretical understanding of matrix optimizers, and hyperparameter transfer. For an overview of neural network optimization methods, we refer to the recent monograph~\citep{Zhang-2023-Dive}.

\paragraph{More discussions on neural network optimization methods.} Beyond matrix optimizers discussed in the main text, many neural network optimizers operate through vector updates. The classical baseline is SGD with momentum, whose practical relevance in deep learning is tied to initialization and momentum tuning~\citep{Robbins-1951-Stochastic, Polyak-1964-Some, Nesterov-1983-Method, Sutskever-2013-Importance}. Coordinatewise adaptive methods instead maintain diagonal statistics of past gradients: AdaGrad and RMSProp accumulate squared-gradient information, Adam combines first- and second-moment exponential averages, AdamW decouples weight decay, Adafactor reduces optimizer memory through factored second-moment estimates, and more recent variants such as Adan and MARS modify the momentum or variance-reduction component~\citep{Duchi-2011-Adaptive, Tieleman-2012-Lecture, Kingma-2015-Adam, Shazeer-2018-Adafactor, Loshchilov-2019-Decoupled, Xie-2024-Adan, Yuan-2025-Mars}. Theoretical analyses of vector optimizers are extensive but optimizer-specific. Indeed, SGD-type methods have nonconvex upper and lower bounds under standard smoothness~\citep{Ghadimi-2013-Stochastic, Arjevani-2023-Lower}. Adaptive methods have been analyzed under relaxed, coordinatewise, or anisotropic smoothness assumptions and through sign-magnitude interpretations of updates~\citep{Reddi-2018-Convergence, Balles-2018-Dissecting, Zhang-2022-Adam, Wang-2023-Convergence, Li-2023-Convergence, Liu-2025-Adagrad, Jiang-2025-Provable, Li-2025-Convergence}. Layerwise-scaled methods such as LARS and LAMB control update scale, which is useful in large-batch regimes~\citep{You-2017-Large, You-2020-Large}. Finally, recent works have framed optimizer design through explicit norm control: modular-norm and operator-norm perspectives motivate layerwise updates based on input-output norms and hyperparameter transfer, the Scion framework unifies normalized, sign-based, and spectral updates as norm-ball steps, and modern constrained optimization views connect many existing methods to implicit or explicit norm constraints~\citep{Large-2024-Scalable, Bernstein-2024-Old, Bernstein-2025-Modular, Pethick-2025-Training, Xie-2024-Implicit, DAngelo-2024-Why, Pethick-2025-Sam}.

\paragraph{More discussions on optimization under heavy-tailed noise.} Most existing work focuses on vector optimizers and can be grouped into clipping, normalization, and coordinatewise sign updates. Clipping-based stochastic methods established the $O(\epsilon^{-\frac{3p-2}{p-1}})$ rate for smooth nonconvex objectives and later extended this to high-probability guarantees, nonsmooth and convex problems, and variational inequalities \citep{Zhang-2020-Adaptive, Cutkosky-2021-High, Liu-2023-Breaking, Nguyen-2023-Improved, Sadiev-2023-High, Gorbunov-2024-High, Liu-2024-High}, while matching lower-bound refinements further clarified the dependence on the initial gap, smoothness, and noise scale \citep{Zhang-2020-Adaptive, Liu-2025-Nonconvex}. A second line shows that clipping is not the only robustification mechanism: momentum-based normalized SGD attains an $O(\epsilon^{-\frac{2p}{p-1}})$ rate when $p$ is unknown \citep{Liu-2025-Nonconvex, Hubler-2025-Gradient, Sun-2025-Revisiting}. Coordinatewise sign methods provide another non-Euclidean route, beginning with SignSGD and its majority-vote and error-feedback variants \citep{Bernstein-2018-Signsgd, Bernstein-2019-Signsgd, Karimireddy-2019-Error, Safaryan-2021-Stochastic, Sun-2023-Momentum}. In particular, the Lion optimizer combines sign updates with two momentum parameters and decoupled weight decay, with analyses via constrained dynamics and stochastic Frank-Wolfe interpretations~\citep{Chen-2023-Symbolic, Chen-2024-Lion, Sfyraki-2025-Lions}. Heavy-tailed modeling also connects this optimization literature to robust learning, online learning, and bandits~\citep{Bubeck-2013-Bandits, Hsu-2014-Heavy, Zhang-2018-Regression, Xue-2021-Nearly, Vural-2022-Mirror, Zhang-2022-Parameter, Xue-2023-Efficient, Ye-2025-Catoni, Liu-2025-Online}, and to empirical studies of neural-network and language-model gradient statistics, Zipfian imbalance, and sign-like adaptivity \citep{Simsekli-2019-Tail, Gurbuzbalaban-2021-Heavy, Kunstner-2023-Noise, Kunstner-2024-Heavy, Kunstner-2025-Scaling, Yadav-2025-Provable}.

\paragraph{Theoretical understanding of matrix optimizers.} Existing theory has developed along two basic directions: convergence analysis and mechanistic interpretation. On the former side, early analyses of Muon study an idealized polar-step version, replacing the finite Newton–Schulz (NS) orthogonalization by an exact matrix sign~\citep{Li-2025-Note, Shen-2025-Convergence, Sato-2025-Convergence}. The work of~\citet{Kim-2026-Convergence} is particularly relevant since it narrows this gap: it proves nonconvex convergence for practical Muon with a finite number of NS steps, shows that the NS error only introduces a multiplicative factor relative to the exact-polar rate, and proves that this factor approaches one doubly exponentially in the number of NS iterations and improves with the polynomial degree. Its comparison with SGD with momentum identifies a rank-dependence advantage for matrix orthogonalization under a nuclear-norm optimality criterion. Closely related works analyze Scion or Muon-like methods as norm-constrained linear minimization steps~\citep{Pethick-2025-Training, Sfyraki-2025-Lions}, but they do not provide the sharp dimension dependence of stochastic optimization in spectral geometry under heavy-tailed noise. A second line of work aims to understand why matrix-sign updates are useful. In particular, the normalized steepest descent view identifies Muon as a spectral-geometry analogue of sign or normalized gradient descent \citep{Bernstein-2024-Old, Chen-2025-Muon, Kovalev-2025-Understanding, Riabinin-2025-Gluon}. Other works focus on implicit bias, simplicity bias, nonsmooth analysis, local quadratic models, nonconvex matrix factorization problems, or preconditioning interpretations \citep{Fan-2025-Implicit, Dragutinovic-2026-To, Gronich-2026-Implicit, Davis-2025-Spectral, Lau-2025-Polargrad, Su-2025-Isotropic, Ma-2026-Preconditioning, Gong-2026-Aro, Gonon-2026-Insights, Parshakova-2026-Muon}.

\paragraph{Hyperparameter transfer.} Hyperparameter transfer seeks scaling rules under which hyperparameters tuned on small proxy models, such as initialization, learning rates, residual scales, and regularization, remain near-optimal as width, depth, or compute grows. The maximal-update parametrization $\mu$P gives the canonical width rule: scale initialization and layerwise learning rates so all layers maintain stable, nontrivial feature learning in the infinite-width limit, enabling zero-shot transfer~\citep{Yang-2021-Feature,Yang-2021-Tuning}. A finite-width complement is the spectral condition, which preserves layerwise input-output geometry by requiring weight updates to have $\Theta(\sqrt{d_{\mathrm{out}}/d_{\mathrm{in}}})$ scale in operator norm~\citep{Yang-2023-Spectral}. Transfer also extends to depth, where residual branches and learning rates must be co-scaled to sustain feature learning~\citep{Dey-2025-Lazy}. Finally, learning rate transfer alone is insufficient for compute-optimal training: regularization can depend on model size, with inverse-width scaling often improving transfer~\citep{Xiao-2024-Rethinking, Qiu-2025-Hyperparameter}. Thus, optimizer comparisons should transfer initialization, learning rate, depth scaling, and regularization rules jointly.

\section{Missing Proofs}
\label{app:proofs}
Throughout this section, we define $r:=\min\{m,n\}$ and $[N]:=\{1,2,\ldots,N\}$.
\subsection{Proof of Lemma~\ref{lemma:dimension-upper-bound}}
We write $\tau(\|\cdot\|):=\tau(\|\cdot\|,m,n,p)$. Letting $\|\cdot\|,\|\cdot\|'$ be norms on $\br^{m\times n}$, if $\|\cdot\|'\le\|\cdot\|\le C\|\cdot\|'$ for some $C\ge1$, by definition (cf. Eq.~\eqref{def:factor}), we have
\begin{equation}
\label{eq:tau-comparison}
\tau(\|\cdot\|)\le C\tau(\|\cdot\|').
\end{equation}
\paragraph{General norm.}
John's theorem~\cite[Theorem~3.1]{Ball-1997-Elementary} gives a Euclidean norm $|\cdot|$ with $|\cdot|\le\|\cdot\|\le\sqrt{mn}|\cdot|$. Lemma~4.3 of \cite{Liu-2025-Nonconvex} yields $\mathbb E|\sum_{t=1}^T Z_t|\le2\sqrt2\mathbb E(\sum_{t=1}^T|Z_t|^p)^{1/p}$ for every martingale difference sequence, that is, $\tau(|\cdot|)\le2\sqrt2$. Hence $\tau(\|\cdot\|)\le2\sqrt{2mn}$ by Eq.~\eqref{eq:tau-comparison}. Conversely, $T=1$, $Z_1=\epsilon A$ with $A\ne0$, and a fair sign $\epsilon$ make the numerator and the denominator of Eq.~\eqref{def:factor} both equal $\|A\|$, and hence every norm satisfies $\tau(\|\cdot\|)\ge1$.
\paragraph{Nuclear norm.} For the lower bound, let $Z_k=\epsilon_k E_{kk}$ for $k\in[r]$, where $E_{kk}\in\mathbb{R}^{m\times n}$ has the $(k,k)^\textnormal{th}$ entry $1$ and the other entries $0$, and the signs $\epsilon_k$ are independent and fair. We then have $\|\sum_k Z_k\|_\textnormal{nuc}=r$ and $(\sum_k\|Z_k\|_\textnormal{nuc}^p)^{1/p}=r^{1/p}$, and hence $\tau(\|\cdot\|_\textnormal{nuc})\ge r^{1-1/p}$.
 
For the upper bound, let $\|\cdot\|_{S_p}$ denote the Schatten $p$-norm. Applying the monotonicity of $\ell_q$ norms and H\"older's inequality, we have
\begin{equation}
\label{eq:nuclear-sp-comparison}
\|\cdot\|_{S_p}\le\|\cdot\|_\textnormal{nuc}\le r^{1-1/p}\|\cdot\|_{S_p}.
\end{equation}
By \cite{Ball-1994-Sharp}, $\mathbb{R}^{m\times n}$ equipped with $\|\cdot\|_{S_p}$, denoted by $S_p^{m,n}$, is $p$-uniformly smooth for every $p\in(1,2]$, and hence \cite[Proposition~10.31(i), Corollary~10.23, Theorem~10.60]{Pisier-2016-Martingales} provide a constant $C_p>0$, depending only on $p$, such that every $S_p^{m,n}$-valued martingale $f$ satisfies $\mathbb E[\sup_{k\ge0}\|f_k\|_{S_p}]\le C_p\mathbb E(\|f_0\|_{S_p}^p+\sum_{k\ge0}\|f_{k+1}-f_k\|_{S_p}^p)^{1/p}$. By Eq.~\eqref{def:factor}, $\tau(\|\cdot\|_{S_p})\le C_p$. Combining this with Eqs.~\eqref{eq:tau-comparison} and \eqref{eq:nuclear-sp-comparison} gives $\tau(\|\cdot\|_\textnormal{nuc})\le C_p r^{1-1/p}$. This completes the proof.
\subsection{Proof of Lemma~\ref{lem:polynomial-inexact-lmo}}
We fix implementations of Newton--Schulz and PolarExpress below. Both start from the Frobenius normalization of the input, where $0/0=0$, and approximate the matrix sign by an odd matrix polynomial. Algorithm~\ref{alg:ns} uses the classical degree-five Newton--Schulz polynomial in every step. Algorithm~\ref{alg:pe} uses the step-dependent PolarExpress polynomials of \citet{Amsel-2026-Polar}, generated by the exact-arithmetic greedy minimax recursion of \citet[Theorem~3.1]{Amsel-2026-Polar} with degree 5. The practical algorithm of \citet[Algorithm~1]{Amsel-2026-Polar} includes additional heuristics for numerical stability. Theorems~3.1 and~3.3 of \citet{Amsel-2026-Polar} are not stated for this variant, and hence our bound on $\delta_\textnormal{PE}$ does not cover it.
\begin{algorithm2e}[h]
\caption{Newton--Schulz}\label{alg:ns}
\LinesNumbered
\SetAlgoVlined
\KwIn{$X\in\mathbb{R}^{m\times n}$ and $K\ge1$.}
$Y_0=X/\|X\|_F$\;
\For{$k=0,\ldots,K-1$}{
$Y_{k+1}=\tfrac18Y_k\bigl(15I_n-10Y_k^\top Y_k+3(Y_k^\top Y_k)^2\bigr)$\;}
\KwOut{$\operatorname{lmo}_\delta(X)=-Y_K$.}
\end{algorithm2e}
\begin{algorithm2e}[h]
\caption{PolarExpress}\label{alg:pe}
\LinesNumbered
\SetAlgoVlined
\KwIn{$X\in\mathbb{R}^{m\times n}$, $K\ge1$, and $\ell\in(0,1)$.}
\lIf{$X=0$}{output $0$}
$\ell_0=\ell$, $u_0=1$\;
\For{$k=0,\ldots,K-1$}{
$(a_k,b_k,c_k)=\argmin_{(a,b,c)\in\mathbb{R}^3}\max_{x\in[\ell_k,u_k]}|1-ax-bx^3-cx^5|$\;
$\ell_{k+1}=a_k\ell_k+b_k\ell_k^3+c_k\ell_k^5$, $u_{k+1}=2-\ell_{k+1}$\;
}
$Y_0=X/\|X\|_F$\;
\For{$k=0,\ldots,K-1$}{
$Y_{k+1}=Y_k\bigl(a_kI_n+b_kY_k^\top Y_k+c_k(Y_k^\top Y_k)^2\bigr)$\;
}
\KwOut{$\operatorname{lmo}_\delta(X)=-Y_K/(1+e_K)$, where $e_K:=(1-\ell^2)^{3^K}$.}
\end{algorithm2e}

Let $X/\|X\|_F=U\operatorname{Diag}(x_1,\ldots,x_r)V^\top$ be an SVD. Thus, $x_i\in[0,1]$ and $\sum_ix_i^2=1$. Let $\phi(x):=\tfrac18(15x-10x^3+3x^5)$ from Algorithm~\ref{alg:ns} and $\psi_k(x):=a_kx+b_kx^3+c_kx^5$ from Algorithm~\ref{alg:pe}. By induction on $k$, both algorithms output $\operatorname{lmo}_\delta(X)=-U\operatorname{Diag}(z(x_1),\ldots,z(x_r))V^\top$, where $z=\phi^{\circ K}$ in Algorithm~\ref{alg:ns} and $z=\frac{\psi_{K-1}\circ\cdots\circ\psi_0}{1+e_K}$ in Algorithm~\ref{alg:pe}. Thus, we have to prove (i) $z$ maps $[0,1]$ into $[0,1]$ and (ii) $\sum_ix_iz(x_i)\ge(1-\delta)\sum_ix_i$. Indeed, if (i) holds, we will have $\|\operatorname{lmo}_\delta(X)\|_\textnormal{op}\le1$. If (ii) holds, since $-\langle X,\operatorname{lmo}_\delta(X)\rangle=\|X\|_F\sum_ix_iz(x_i)$ and $\|X\|_\textnormal{nuc}=\|X\|_F\sum_ix_i$, we will have $-\langle X,\operatorname{lmo}_\delta(X)\rangle\geq (1-\delta)\|X\|_\textnormal{nuc}$.

\paragraph{Newton--Schulz.} We prove (i). On $[0,1]$ we have $\phi'(x)=\tfrac{15}{8}(1-x^2)^2\ge0$, $\phi(0)=0$, and $\phi(1)=1$, and hence $\phi$ maps $[0,1]$ into $[0,1]$, and so does $z=\phi^{\circ K}$. We prove (ii). For $x\in(0,1]$,
\begin{equation*}
\phi(x)-\tfrac{15x}{8+7x}=\tfrac{x^2(1-x)(105+25x-45x^2-21x^3)}{8(8+7x)}\ge0,
\end{equation*}
since $105+25x-45x^2-21x^3\ge105-45-21>0$. Hence, $\phi(x)^{-1}-1\le\tfrac{8}{15}(x^{-1}-1)$. Applying it $K$ times gives $z(x)^{-1}-1\le(8/15)^K(x^{-1}-1)$, and hence $z(x)\ge x/(x+(8/15)^K)$ for all $x\in[0,1]$. By the Cauchy--Schwarz inequality,
\begin{equation*}
\sum_ix_iz(x_i)\ge\sum_i\tfrac{x_i^2}{x_i+(8/15)^K}\ge\tfrac{(\sum_ix_i)^2}{\sum_ix_i+r(8/15)^K}.
\end{equation*}
Since $\sum_ix_i\ge(\sum_ix_i^2)^{1/2}=1$, (ii) holds with $\delta_\textnormal{NS}=\frac{r(8/15)^K}{1+r(8/15)^K}\in[0,1)$, and $\delta_\textnormal{NS}\le r(8/15)^K$.
 
\paragraph{PolarExpress.} We run Algorithm~\ref{alg:pe} with $\ell=\sqrt{K}3^{-K/2}$ if $r\sqrt{K}3^{-K/2}<1$, and with $\ell=1/(2r)$ otherwise. In both cases $r\ell<1$. We write $\psi:=\psi_{K-1}\circ\cdots\circ\psi_0$ so that $z=\psi/(1+e_K)$.
 
We prove (i). Theorems~3.1 and~3.3 of \citet{Amsel-2026-Polar} give $|\psi-1|\le e_K$ on $[\ell,1]$, and the proof of \citet[Lemma~C.2]{Amsel-2026-Polar} gives $0\le\psi\le1+e_K$ on $[0,\ell]$. Hence, $z$ maps $[0,1]$ into $[0,1]$. We prove (ii). By (i), $z\ge\frac{1-e_K}{1+e_K}$ on $[\ell,1]$ and $z\ge0$ on $[0,\ell]$. Since $\sum_{x_i<\ell}x_i\le r\ell\le r\ell\sum_ix_i$, we have
\begin{equation*}
\sum_ix_iz(x_i)\ge\tfrac{1-e_K}{1+e_K}\sum_{x_i\ge\ell}x_i\ge(1-r\ell)\tfrac{1-e_K}{1+e_K}\sum_ix_i.
\end{equation*}
Thus, (ii) holds with
\begin{equation*}
\delta_\textnormal{PE}=1-(1-r\ell)\tfrac{1-e_K}{1+e_K}=\tfrac{r\ell(1-e_K)+2e_K}{1+e_K}\le r\ell+2e_K.
\end{equation*}
If $r\sqrt{K}3^{-K/2}<1$, we have $e_K\le\exp(-3^K\ell^2)=\exp(-K)$, and hence $\delta_\textnormal{PE}\le r\sqrt{K}3^{-K/2}+2\exp(-K)$. If $r\sqrt{K}3^{-K/2}\geq1$, Assumption~\ref{assumption:inexactness} holds since $r\ell<1$ and $e_K<1$ give $\delta_\textnormal{PE}=1-(1-r\ell)\frac{1-e_K}{1+e_K}<1$, completing the proof.
\subsection{Proof of Theorem~\ref{thm:main}}
Without loss of generality, we assume $m\leq n$. We define the chain function. Let
\begin{equation*}
\Psi(u):=\begin{cases}0,&u\leq\tfrac{1}{2},\\ \exp(1-(2u-1)^{-2}),&u>\tfrac{1}{2},\end{cases}
\quad \Phi(u):=\sqrt{e}\int_{-\infty}^u e^{-t^2/2}dt,
\end{equation*}
and for $K\in\mathbb{N}$ define $f_K:\mathbb{R}^K\to\mathbb{R}$ by
\begin{equation*}
f_K(z):=-\Psi(1)\Phi(z_1)+\textstyle\sum_{j=2}^K[\Psi(-z_{j-1})\Phi(-z_j)-\Psi(z_{j-1})\Phi(z_j)].
\end{equation*}
We say that coordinate $j$ of $z\in\mathbb{R}^K$ is activated if $|z_j|>\frac12$, and we define the progress of $z$ as
\begin{equation*}
\operatorname{prog}_{\frac{1}{2}}(z):=\max(\{j\in[K]:|z_j|>\tfrac{1}{2}\}\cup\{0\}).
\end{equation*}
Let $\Delta_0=12$, $\gamma_\infty=23$, and $\ell_1=152$. Lemma~\ref{lemma:stochastic-chain} is proved in \citet{Arjevani-2023-Lower}. Items~2 and~5 also follow directly from the definition, since $\Psi,\Psi',\Phi,\Phi'\geq0$ and $\Psi$ vanishes on $(-\infty,\frac12]$.
\begin{lemma}\label{lemma:stochastic-chain}
For any $K\in\mathbb{N}$ and any $z\in\mathbb{R}^K$, we have the following.
\begin{enumerate}
\item $f_K(0)-\inf_{u\in\mathbb{R}^K}f_K(u)\leq\Delta_0 K$.
\item $-\gamma_\infty\leq\partial_jf_K(z)\leq0$ for every $j\in[K]$, and $\|\nabla^2f_K(z)\|_\textnormal{op}\leq\ell_1$.
\item If $j$ is the least index with $|z_j|\leq1$, we have $\partial_j f_K(z)\leq-1$.
\item If $k:=\operatorname{prog}_{\frac{1}{2}}(z)$, we have $\partial_j f_K(z)=0$ for $k+2\leq j\leq K$.
\item The partial derivative $\partial_j f_K(z)$ depends only on the coordinates $z_{j-1},z_j,z_{j+1}$ that exist. If $|z_j| \leq \frac{1}{2}$, then it depends only on
$z_{j-1}$ and $z_j$, with the same boundary convention.
\end{enumerate}
\end{lemma}
We define the hard function. Let $1\leq K\leq n$, $M:=\max\{m,K\}$, and $N:=\lfloor n/M\rfloor\geq1$. We split the first $MN$ columns into $M$ consecutive blocks of $N$ columns. For $i\in[m]$, $j\in[K]$, and $a\in[N]$, let
\begin{equation*}
c_{i,j}^a:=N((i+j-2)\bmod M)+a
\end{equation*}
be the $a^\textnormal{th}$ entry of block $(i+j-2)\bmod M$. Then, $X_{i,c_{i,j}^a}$ is assigned to chain coordinate $j$.
\begin{lemma}\label{lemma:layout}
The column indices $c_{i,j}^a$ satisfy the following.
\begin{enumerate}
\item For each $i\in[m]$, the indices $c_{i,j}^a$ are pairwise distinct over $(j,a)\in[K]\times[N]$.
\item For each $j\in[K]$, the sets $\{c_{i,j}^a:a\in[N]\}$ are pairwise disjoint over $i\in[m]$.
\item The entries $(i,c_{i,j}^a)$ are pairwise distinct over $(i,j,a)\in[m]\times[K]\times[N]$.
\end{enumerate}
\end{lemma}

Lemma~\ref{lemma:layout} is trivial. For $X\in\mathbb{R}^{m\times n}$, $\theta\in\{-1,1\}^{m\times K\times N}$ and $A,\rho>0$, we define
\begin{equation*}
y_{i,j}^a(X):=\rho^{-1}\sqrt{N}X_{i,c_{i,j}^a},\ q_j^\theta(X):=\tfrac{4}{mN}\textstyle\sum_{i=1}^m\textstyle\sum_{a=1}^N\theta_{i,j}^a\tanh y_{i,j}^a(X),\ q^\theta(X):=(q_1^\theta(X),\ldots,q_K^\theta(X)).
\end{equation*}
The hard function is $F_{\theta,K}$, where for $k\in\{0,1,\ldots,K\}$ we define
\begin{align*}
F_{\theta,k}(X)&:=A f_K(q_1^\theta(X),\ldots,q_k^\theta(X),0,\ldots,0)+R(X),\\
R(X)&:=\tfrac{8A}{mN}\textstyle\sum_{r=1}^K\textstyle\sum_{i=1}^m\textstyle\sum_{a=1}^N\log\cosh y_{i,r}^a(X).
\end{align*}
$R$ ensures that the gradient is large whenever $|q_j^\theta(X)|\leq1$ for some $j$ (Lemma~\ref{lemma:stochastic-matrix}). We write $\theta_j:=(\theta_{i,j}^a)_{i\in[m],a\in[N]}$, $\theta_{\leq k}:=(\theta_1,\ldots,\theta_k)$, and $\operatorname{prog}(X):=\operatorname{prog}_{\frac12}(q^\theta(X))$. Let $C_s:=16\ell_1+8\gamma_\infty+8$.

\begin{lemma}\label{lemma:stochastic-matrix}
Let $\Delta,L,\epsilon>0$, $A/\rho=4\epsilon$, $C_{s}A/\rho^2\leq L$, and $\Delta_0 AK\leq\Delta$. If $\theta\in\{-1,1\}^{m\times K\times N}$, $F_{\theta,K}\in\mathcal{F}_\textnormal{op}(m,n,\Delta,L)$. If, in addition, $\min_{j\in[K]}|q_j^\theta(X)|\leq1$, we have $\|\nabla F_{\theta,K}(X)\|_\textnormal{nuc}>2\epsilon$.
\end{lemma}
\begin{proof}
We first bound the initial gap. Since $\log\cosh\geq0$ and $\log\cosh0=0$, Lemma~\ref{lemma:stochastic-chain} gives $F_{\theta,K}(0)-\inf F_{\theta,K}\leq A(f_K(0)-\inf f_K)\leq\Delta_0AK\leq\Delta$.

We show smoothness. Letting $X,H\in\mathbb{R}^{m\times n}$, and $y_{i,j}^a=y_{i,j}^a(X)$ and $w_{i,j}^a:=H_{i,c_{i,j}^a}$, we have
\begin{align*}
\tfrac{\partial q_j^\theta}{\partial X_{i,c_{i,j}^a}}&=\tfrac{4\theta_{i,j}^a}{m\sqrt{N}\rho}\operatorname{sech}^2y_{i,j}^a,\quad
\tfrac{\partial^2 q_j^\theta}{\partial X_{i,c_{i,j}^a}^2}=-\tfrac{8\theta_{i,j}^a}{m\rho^2}\operatorname{sech}^2y_{i,j}^a\tanh y_{i,j}^a,\\
\tfrac{\partial^2 R}{\partial X_{i,c_{i,j}^a}^2}&=\tfrac{8A}{m\rho^2}\operatorname{sech}^2y_{i,j}^a,
\end{align*}
and all mixed second partial derivatives of $q_j^\theta$ and $R$ vanish. Using $|\operatorname{sech}^2|\leq1$, $|\operatorname{sech}^2\tanh|\leq1$, the Cauchy--Schwarz inequality over the $mN$ terms of each $q_j^\theta$, and $\sum_{i,j,a}(w_{i,j}^a)^2\leq\|H\|_F^2$ (Lemma~\ref{lemma:layout}, item~3), we obtain
\begin{align*}
\|\mathcal{J}q^\theta(X)H\|_2^2&=\textstyle\sum_{j=1}^K\left(\sum_{i,a}\tfrac{4\theta_{i,j}^a}{m\sqrt{N}\rho}\operatorname{sech}^2(y_{i,j}^a)w_{i,j}^a\right)^2\\
&\leq\textstyle\sum_{j=1}^K\tfrac{16}{m^2N\rho^2}\cdot mN\sum_{i,a}(w_{i,j}^a)^2\leq\tfrac{16}{m\rho^2}\|H\|_F^2,\\
\textstyle\sum_{j=1}^K\left|\langle H,\nabla^2q_j^\theta(X)H\rangle\right|&\leq\textstyle\sum_{i,j,a}\tfrac{8}{m\rho^2}(w_{i,j}^a)^2\leq\tfrac{8}{m\rho^2}\|H\|_F^2,\\
\langle H,\nabla^2R(X)H\rangle&\leq\textstyle\sum_{i,j,a}\tfrac{8A}{m\rho^2}(w_{i,j}^a)^2\leq\tfrac{8A}{m\rho^2}\|H\|_F^2.
\end{align*}
By the chain rule and Lemma~\ref{lemma:stochastic-chain}, item~2, we therefore have
\begin{align*}
\left|\langle H,\nabla^2F_{\theta,K}(X)H\rangle\right|&=\Big|A\langle\mathcal{J}q^\theta(X)H,\nabla^2f_K(q^\theta(X))\mathcal{J}q^\theta(X)H\rangle\\
&\quad+A\textstyle\sum_{j=1}^K\partial_jf_K(q^\theta(X))\langle H,\nabla^2q_j^\theta(X)H\rangle+\langle H,\nabla^2R(X)H\rangle\Big|\\
&\leq A\ell_1\|\mathcal{J}q^\theta(X)H\|_2^2+A\gamma_\infty\textstyle\sum_{j=1}^K\left|\langle H,\nabla^2q_j^\theta(X)H\rangle\right|+\langle H,\nabla^2R(X)H\rangle\\
&\leq\tfrac{A}{m\rho^2}(16\ell_1+8\gamma_\infty+8)\|H\|_F^2=\tfrac{C_sA}{m\rho^2}\|H\|_F^2\leq\tfrac{L}{m}\|H\|_F^2.
\end{align*}
Since $\nabla^2F_{\theta,K}(X)$ is symmetric, $\|\nabla^2F_{\theta,K}(X)\|_{F\to F}\leq\frac{L}{m}$, and thus $\|\nabla F_{\theta,K}(X)-\nabla F_{\theta,K}(Y)\|_F\leq\frac{L}{m}\|X-Y\|_F$ for all $X,Y$. Using $\|\cdot\|_\textnormal{nuc}\leq\sqrt{m}\|\cdot\|_F$ and $\|\cdot\|_F\leq\sqrt{m}\|\cdot\|_\textnormal{op}$ on $\mathbb{R}^{m\times n}$, we obtain
\begin{equation*}
\|\nabla F_{\theta,K}(X)-\nabla F_{\theta,K}(Y)\|_\textnormal{nuc}\leq\sqrt{m}\cdot\tfrac{L}{m}\|X-Y\|_F\leq L\|X-Y\|_\textnormal{op}.
\end{equation*}

We show that $\|\nabla F_{\theta,K}(X)\|_\textnormal{nuc}>2\epsilon$. Let $j$ be the least index with $|q_j^\theta(X)|\leq1$, and let $b:=\partial_jf_K(q^\theta(X))$. By Lemma~\ref{lemma:stochastic-chain}, item~3, $b\leq-1$. Let $W\in\mathbb{R}^{m\times n}$ have entries $W_{i,c_{i,j}^a}:=\operatorname{sign}([\nabla F_{\theta,K}(X)]_{i,c_{i,j}^a})/\sqrt{N}$ and zero elsewhere. By Lemma~\ref{lemma:layout}, item~2, the rows of $W$ have disjoint supports, and each row has at most $N$ nonzero entries of absolute value $1/\sqrt{N}$. Hence $WW^\top\preceq I_m$, and thus $\|W\|_\textnormal{op}\leq1$. By duality, we have
\begin{equation*}
\|\nabla F_{\theta,K}(X)\|_\textnormal{nuc}\geq\langle \nabla F_{\theta,K}(X),W\rangle=\tfrac{1}{\sqrt{N}}\textstyle\sum_{i=1}^m\textstyle\sum_{a=1}^N|[\nabla F_{\theta,K}(X)]_{i,c_{i,j}^a}|.
\end{equation*}
By Lemma~\ref{lemma:layout}, item~3, the entry $X_{i,c_{i,j}^a}$ appears only in $q_j^\theta$ and in one term of $R$. Let $u_i^a:=\theta_{i,j}^a\tanh y_{i,j}^a(X)\in[-1,1]$. Using $\operatorname{sech}^2=1-\tanh^2$, we obtain
\begin{equation*}
\theta_{i,j}^a[\nabla F_{\theta,K}(X)]_{i,c_{i,j}^a}=\tfrac{A}{m\sqrt{N}\rho}z_b(u_i^a),\quad \textnormal{where}\quad z_b(u):=4b(1-u^2)+8u.
\end{equation*}
Let $\kappa:=\sqrt{2}-1$. Since $b<0$, $z_b$ is convex, with $z_b(-1)=-8$ and $z_b(\kappa)=8\kappa(1+b)\leq0$. If $u\in[-1,\kappa]$, convexity gives $z_b(u)\leq-8\frac{\kappa-u}{1+\kappa}\leq-4(\kappa-u)$. If $u>\kappa$, we have $|z_b(u)|\geq4(\kappa-u)$. Hence $|z_b(u)|\geq4(\kappa-u)$ for all $u\in[-1,1]$. Since $\frac{1}{mN}\sum_{i,a}u_i^a=\frac14q_j^\theta(X)\leq\frac14$ and $A/\rho=4\epsilon$, we obtain
\begin{align*}
\|\nabla F_{\theta,K}(X)\|_\textnormal{nuc}&\geq\tfrac{A}{mN\rho}\textstyle\sum_{i,a}|z_b(u_i^a)|\geq\tfrac{4A}{\rho}\left(\kappa-\tfrac{1}{mN}\textstyle\sum_{i,a}u_i^a\right)\geq\tfrac{4A}{\rho}\left(\kappa-\tfrac14\right)>2\epsilon.
\end{align*}
This completes the proof.
\end{proof}

We now define the oracle. We fix $p\in(1,2]$ and suppose $\epsilon\leq\sigma_0/2000$. For each $k\in\{0,1,\ldots,K\}$, let $q^{\leq k}(X):=(q_1^\theta(X),\ldots,q_k^\theta(X),0,\ldots,0)\in\mathbb{R}^K$. The gradient of the function truncated after $k$ chain coordinates is
\begin{equation*}
\nabla F_{\theta,k}(X)=A\textstyle\sum_{\ell=1}^k\partial_\ell f_K(q^{\leq k}(X))\nabla q_\ell^\theta(X)+\nabla R(X).
\end{equation*}
For each $k\in\{0,1,\ldots,K-1\}$, let
\begin{align*}
H_k^\theta(X):=\nabla F_{\theta,k+1}(X)-\nabla F_{\theta,k}(X),\quad J_k:=\begin{cases}\{1\},&k=0,\\\{k,k+1\},&k\geq1.\end{cases}
\end{align*}
For $i\in[m]$, $\ell\in[K]$, and $v\in\mathbb{R}^N$, let $\mathcal{E}_{i,\ell}(v)\in\mathbb{R}^{m\times n}$ place $v_a$ at entry $(i,c_{i,\ell}^a)$ for $a\in[N]$ and zero elsewhere. Then $\mathcal{E}_{i,\ell}(v)$ has a single nonzero row, and hence $\|\mathcal{E}_{i,\ell}(v)\|_\textnormal{nuc}=\|v\|_2$.

We set $\nu:=(2000\epsilon/\sigma_0)^{\frac{p}{p-1}}\in(0,1]$. The oracle noise is $\xi=(\chi,I,V,g)$, whose components are independent, with $\chi\sim\operatorname{Bernoulli}(\nu)$, $I$ uniform on $[m]$, $V$ uniform on $[0,1]$, and $g\sim\mathcal{N}(0,I_N)$. The distribution of $\xi$ does not depend on $k$, $\theta$, or $X$. For each $k\in\{0,1,\ldots,K-1\}$, let
\begin{equation*}
L_k:=\begin{cases}1,&k=0,\\ k,&k\geq1\textnormal{ and }V\leq\frac12,\\ k+1,&k\geq1\textnormal{ and }V>\frac12.\end{cases}
\end{equation*}
Then $L_k$ is uniform on $J_k$, and $(I,L_k)$ is uniform on $[m]\times J_k$. For each $k\in\{0,1,\ldots,K-1\}$, we define
\begin{align*}
h_k&:=\left([H_k^\theta(X)]_{I,c_{I,L_k}^a}\right)_{a=1}^N,\quad Z_k:=\mathcal{E}_{I,L_k}\left(m|J_k|h_k+\tfrac{800\epsilon}{\sqrt{N}}g\right),\\
G_{\theta,k}(X,\xi)&:=\nabla F_{\theta,k}(X)+\tfrac{\chi}{\nu}Z_k,\\
Q_k^{\theta_{\leq k}}(X,\xi)&:=\nabla F_{\theta,k}(X)+\tfrac{\chi}{\nu}\mathcal{E}_{I,L_k}\left(\tfrac{800\epsilon}{\sqrt{N}}g\right),
\end{align*}
and set $G_{\theta,K}(X,\xi):=\nabla F_{\theta,K}(X)$. The oracle is
\begin{equation*}
G_\theta(X,\xi):=G_{\theta,\operatorname{prog}(X)}(X,\xi).
\end{equation*}
At a point $X$ with $k=\operatorname{prog}(X)<K$, the oracle returns $\nabla F_{\theta,k}(X)$, and with probability $\nu$ it adds a rescaled noisy copy of the row $h_k$ of $H_k^\theta(X)$ selected by $(I,L_k)$. The oracle $Q_k^{\theta_{\leq k}}$ is obtained from $G_{\theta,k}$ by replacing $h_k$ with $0$.


\begin{lemma}\label{lemma:stochastic-oracle}
Under the conditions of Lemma~\ref{lemma:stochastic-matrix} and $\epsilon\leq\sigma_0/2000$, the following hold.
\begin{enumerate}
\item For every $X$, $\mathbb{E}G_\theta(X,\xi)=\nabla F_{\theta,K}(X)$ and $\mathbb{E}\|G_\theta(X,\xi)-\nabla F_{\theta,K}(X)\|_\textnormal{nuc}^p\leq\sigma_0^p$.
\item If $k<K$, $G_{\theta,k}$ depends on $\theta$ only through $\theta_{\leq k+1}$, and $Q_k^{\theta_{\leq k}}$ depends on $\theta$ only through $\theta_{\leq k}$.
\item For every $\theta$, $X$, and $0\leq k<K$, we have $\operatorname{KL}(G_{\theta,k}(X,\xi)\|Q_k^{\theta_{\leq k}}(X,\xi))\leq N\nu/2$.
\end{enumerate}
\end{lemma}
\begin{proof}
We first show item 2. By the formula for $\nabla F_{\theta,k}$, $\nabla F_{\theta,k}$ involves only $q_1^\theta,\ldots,q_k^\theta$, and $R$ does not involve $\theta$. Hence $\nabla F_{\theta,k}$ and $Q_k^{\theta_{\leq k}}$ depend only on $\theta_{\leq k}$, while $H_k^\theta$ and $G_{\theta,k}$ depend only on $\theta_{\leq k+1}$.

We next show that $H_k^\theta(X)$ is supported on $(i,c_{i,j}^a)$ for $j\in J_k, i\in[m],a\in[N]$, for every $X$ and $0\leq k<K$. By definition, we have
\begin{equation*}
H_k^\theta(X)=A\partial_{k+1}f_K(q^{\leq k+1}(X))\nabla q_{k+1}^\theta(X)+A\textstyle\sum_{\ell=1}^k\left[\partial_\ell f_K(q^{\leq k+1}(X))-\partial_\ell f_K(q^{\leq k}(X))\right]\nabla q_\ell^\theta(X).
\end{equation*}
The points $q^{\leq k+1}(X)$ and $q^{\leq k}(X)$ differ only in coordinate $k+1$. By Lemma~\ref{lemma:stochastic-chain}, item~5, the bracket vanishes for $\ell\leq k-1$. Since $\nabla q_\ell^\theta(X)$ is supported on $(i,c_{i,\ell}^a)$ for $\ell\in J_k$, the claim follows.

We show unbiasedness. If $\operatorname{prog}(X)=K$, the oracle is exact. Otherwise, let $k:=\operatorname{prog}(X)<K$. Then $|q_\ell^\theta(X)|\leq\frac12$ for all $\ell>k$. We claim that $\nabla F_{\theta,k+1}(X)=\nabla F_{\theta,K}(X)$. By Lemma~\ref{lemma:stochastic-chain}, item~4, $\partial_\ell f_K(q^\theta(X))=0$ for $\ell\geq k+2$. For $\ell\leq k$, $\partial_\ell f_K$ depends only on coordinates up to $k+1$, and the first $k+1$ coordinates of $q^\theta(X)$ and $q^{\leq k+1}(X)$ are equal. For $\ell=k+1$, $\partial_{k+1}f_K$ depends only on coordinates $k$ and $k+1$, since $|q_{k+1}^\theta(X)|\leq\frac12$. Hence $\partial_\ell f_K(q^\theta(X))=\partial_\ell f_K(q^{\leq k+1}(X))$ for $\ell\leq k+1$, and the claim follows from the formula for $\nabla F_{\theta,k}$. Since $(I,L_k)$ is uniform on $[m]\times J_k$ and $H_k^\theta(X)$ is supported on $(i,c_{i,j}^a)$ for $j\in J_k, i\in[m],a\in[N]$, we have
\begin{equation*}
\mathbb{E}\left[\mathcal{E}_{I,L_k}(m|J_k|h_k)\right]=\textstyle\sum_{i\in[m]}\sum_{\ell\in J_k}\tfrac{1}{m|J_k|}\cdot m|J_k|\mathcal{E}_{i,\ell}\left(([H_k^\theta(X)]_{i,c_{i,\ell}^a})_{a=1}^N\right)=H_k^\theta(X).
\end{equation*}
 Since $\mathbb{E}Z_k=H_k^\theta(X)$, $\mathbb{E}G_\theta(X,\xi)=\nabla F_{\theta,k}(X)+H_k^\theta(X)=\nabla F_{\theta,k+1}(X)=\nabla F_{\theta,K}(X)$.

We show the moment bound. The error is $G_\theta(X,\xi)-\nabla F_{\theta,K}(X)=\frac{\chi}{\nu}Z_k-H_k^\theta(X)$. By the definition of $H_k^\theta$, Lemma~\ref{lemma:stochastic-chain}, item~2, and $A/\rho=4\epsilon$, every entry of $H_k^\theta(X)$ satisfies $|[H_k^\theta(X)]_{i,c_{i,\ell}^a}|\leq A\gamma_\infty\frac{4}{m\sqrt{N}\rho}\leq\frac{400\epsilon}{m\sqrt{N}}$. Since $H_k^\theta(X)$ has at most $2mN$ nonzero entries, and since $Z_k$ has a single nonzero row, we obtain
\begin{align*}
\|H_k^\theta(X)\|_\textnormal{nuc}&\leq\sqrt{m}\|H_k^\theta(X)\|_F\leq\sqrt{m}\cdot\sqrt{2mN}\cdot\tfrac{400\epsilon}{m\sqrt{N}}\leq600\epsilon,\\
m|J_k|\|h_k\|_2&\leq2m\cdot\sqrt{N}\cdot\tfrac{400\epsilon}{m\sqrt{N}}=800\epsilon,\\
\mathbb{E}\|Z_k\|_\textnormal{nuc}^2&=\mathbb{E}\left\|m|J_k|h_k+\tfrac{800\epsilon}{\sqrt{N}}g\right\|_2^2=\mathbb{E}\left\|m|J_k|h_k\right\|_2^2+800^2\epsilon^2\leq2\cdot800^2\epsilon^2.
\end{align*}
Since $\chi$ is independent of $Z_k$, we have $\mathbb{E}|\chi/\nu|^p=\nu^{1-p}$. Minkowski's inequality and $p\leq2$ give
\begin{align*}
\left(\mathbb{E}\|G_\theta(X,\xi)-\nabla F_{\theta,K}(X)\|_\textnormal{nuc}^p\right)^{1/p}&\leq\nu^{-\frac{p-1}{p}}\left(\mathbb{E}\|Z_k\|_\textnormal{nuc}^2\right)^{1/2}+\|H_k^\theta(X)\|_\textnormal{nuc}\\
&\leq800\sqrt{2}\epsilon\nu^{-\frac{p-1}{p}}+600\epsilon\leq2000\epsilon\nu^{-\frac{p-1}{p}}=\sigma_0.
\end{align*}

We show item 3. We fix $\theta$, $X$, and $k$, and write $G:=G_{\theta,k}(X,\xi)$ and $Q:=Q_k^{\theta_{\leq k}}(X,\xi)$. Conditioning on $(\chi,I,V)$, we have, for every measurable set $S$,
\begin{equation*}
\mathbb{P}(G\in S)=\mathbb{E}\left[\mathbb{P}(G\in S\mid\chi,I,V)\right],\quad\mathbb{P}(Q\in S)=\mathbb{E}\left[\mathbb{P}(Q\in S\mid\chi,I,V)\right],
\end{equation*}
where the expectation is over the random variable $(\chi,I,V)$. By the data processing inequality,
\begin{equation}
\label{eq:KLGQ}
\operatorname{KL}(G\|Q)\leq\mathbb{E}\left[\operatorname{KL}\left(\mathbb{P}(G\in\cdot\mid\chi,I,V)\middle\|\mathbb{P}(Q\in\cdot\mid\chi,I,V)\right)\right].
\end{equation}
Conditionally on $\chi=0$, both $G$ and $Q$ equal $\nabla F_{\theta,k}(X)$, and the conditional divergence is $0$. Conditionally on $\chi=1$ and $(I,V)$, $(I,L_k)$ is fixed. We write $s:=\frac{800\epsilon}{\sqrt{N}\nu}$ and $\mu:=\frac{m|J_k|}{\nu}h_k$. Then,
\begin{equation*}
G=\nabla F_{\theta,k}(X)+\mathcal{E}_{I,L_k}(\mu+sg),\quad Q=\nabla F_{\theta,k}(X)+\mathcal{E}_{I,L_k}(sg).
\end{equation*}
The map $v\mapsto \nabla F_{\theta,k}(X)+\mathcal{E}_{I,L_k}(v)$ is injective, and therefore the conditional divergence equals the divergence between $\mathcal{N}(\mu,s^2I_N)$ and $\mathcal{N}(0,s^2I_N)$, which is $\|\mu\|_2^2/(2s^2)$. Since $m|J_k|\|h_k\|_2\leq800\epsilon$, 
\begin{equation*}
\tfrac{\|\mu\|_2^2}{2s^2}\leq\tfrac{N\nu^2}{2\cdot800^2\epsilon^2}\cdot\tfrac{(800\epsilon)^2}{\nu^2}=\tfrac{N}{2}.
\end{equation*}
Since $\mathbb{P}(\chi=1)=\nu$, taking the expectation of Eq.~\eqref{eq:KLGQ} gives $\operatorname{KL}(G\|Q)\leq N\nu/2$.
\end{proof}

We now give the information argument. Fix $T\in\mathbb{N}$ and consider an algorithm $\mathsf{A}$ interacting with an oracle $\mathsf{O}$. Let $\omega$ be its internal randomness, independent of $\theta$, and let $\xi_1,\ldots,\xi_T$ be independent copies of $\xi$, independent of $(\theta,\omega)$. Its first $T$ queries and replies are described by measurable maps $\mathsf{A}_1,\ldots,\mathsf{A}_T$ through
\begin{equation*}
X_t=\mathsf{A}_t(\omega,Y_1,\ldots,Y_{t-1}),\quad Y_t=\mathsf{O}(X_t,\xi_t),\quad t\in[T].
\end{equation*}
The next query point is $X_{T+1}=\mathsf{A}_{T+1}(\omega,Y_1,\ldots,Y_T)$, where $\mathsf{A}_{T+1}$ is measurable. Thus the first $T$ oracle replies determine the query points $X_1,\ldots,X_{T+1}$.

If $X$ is independent of $\theta_j$, then conditionally on $X$ the summands $\theta_{i,j}^a\tanh y_{i,j}^a(X)$ are independent, centered, and bounded by $1$ in absolute value. Hoeffding's inequality then gives
\begin{equation}\label{eq:stochastic-hidden-sign}
\mathbb{P}(|q_j^\theta(X)|>\tfrac12\mid X)\leq2e^{-mN/128}.
\end{equation}
\begin{lemma}
\label{lemma:stochastic-adaptive-output}
There are universal constants $c_\star,C_\star>0$ with the following property. Let $T\in\mathbb{N}$ and let $\theta_{i,j}^a$ be independent and uniform on $\{-1,1\}$. For any randomized first-order algorithm $\mathsf{A}$ accessing $F_{\theta,K}$ only through replies of $G_\theta$, let $X_1,\ldots,X_{T+1}$ denote its first $T+1$ query points. If $mN\geq C_\star\log(2(K+1)(T+1))$ and $T\leq c_\star Km/\nu$, then $\mathbb{P}(\min_{t\in[T+1]}\|\nabla F_{\theta,K}(X_t)\|_\textnormal{nuc}>2\epsilon)\geq\tfrac{3}{4}$.
\end{lemma}
\begin{proof}
For $j\in[K]$, let $E_j$ be the event that coordinate $j$ or a later one is activated at some query, and let $L_j$ be the number of queries at which the progress is $j-1$, namely,
\begin{equation}\label{eq:Ej-def}
E_j:=\{\exists t\in[T+1]:\operatorname{prog}(X_t)\geq j\},\quad L_j:=\textstyle\sum_{t=1}^T\mathbf{1}\{\operatorname{prog}(X_t)=j-1\}.
\end{equation}
By definition, $\sum_{j=1}^KL_j\leq T$. If $\|\nabla F_{\theta,K}(X_t)\|_\textnormal{nuc}\leq2\epsilon$ for some $t\in[T+1]$, Lemma~\ref{lemma:stochastic-matrix} gives $|q_K^\theta(X_t)|>\frac12$, which implies $\operatorname{prog}(X_t)=K$, and hence
\begin{equation}\label{eq:EK-inclusion}
\{\textstyle\min_{t\in[T+1]}\|\nabla F_{\theta,K}(X_t)\|_\textnormal{nuc}\leq2\epsilon\}\subseteq E_K\subseteq E_j\quad\textnormal{for every }j\in[K].
\end{equation}
We take $C_\star=2048$ and $c_\star=1/1024$. The assumption $mN\geq C_\star\log(2(K+1)(T+1))$ gives $\frac{mN}{128}-\log(2K(T+1))\geq\frac{mN}{256}>0$. We claim that, for every $j\in[K]$,
\begin{equation}\label{eq:stochastic-stage-cost}
\mathbb{P}(E_j)\leq\tfrac{\frac{N\nu}{2}\mathbb{E}L_j+\log2}{\frac{mN}{128}-\log(2K(T+1))}.
\end{equation}
Assuming~\eqref{eq:stochastic-stage-cost}, we derive the desired result. By~\eqref{eq:EK-inclusion}, it suffices to prove $\mathbb{P}(E_K)\leq\tfrac{1}{4}$. Averaging~\eqref{eq:stochastic-stage-cost} over $j\in[K]$, and using $E_K\subseteq E_j$ and $\sum_jL_j\leq T$, we obtain
\begin{equation*}
\mathbb{P}(E_K)\leq\tfrac{1}{K}\textstyle\sum_{j=1}^K\mathbb{P}(E_j)\leq\tfrac{\frac{N\nu T}{2K}+\log2}{\frac{mN}{128}-\log(2K(T+1))}.
\end{equation*}
The denominator is at least $mN/256$, and the assumption $T\leq c_\star Km/\nu$ bounds the numerator by $\frac{c_\star mN}{2}+\log2$. Since $mN\geq2048\log4$, the right-hand side is at most $256(\frac{c_\star}{2}+\frac{\log2}{mN})\leq\frac14$.

We now prove~\eqref{eq:stochastic-stage-cost}. Fix $j\in[K]$. We compare the algorithm history under $G_\theta$ with a run in which $\theta_j$ is never revealed, through two auxiliary oracles. Letting $\operatorname{prog}_{<j}(X):=\max(\{\ell\in[j-1]:|q_\ell^\theta(X)|>\tfrac12\}\cup\{0\})$ be the progress computed from the first $j-1$ coordinates, we define
\begin{align*}
\widetilde{G}_j(X,\xi)&:=\begin{cases}Q_{j-1}^{\theta_{\leq j-1}}(X,\xi),&\operatorname{prog}(X)=j-1,\\ G_\theta(X,\xi),&\textnormal{otherwise},\end{cases}\\
G^{(j)}(X,\xi)&:=\begin{cases}G_{\theta,\operatorname{prog}_{<j}(X)}(X,\xi),&\operatorname{prog}_{<j}(X)<j-1,\\ Q_{j-1}^{\theta_{\leq j-1}}(X,\xi),&\operatorname{prog}_{<j}(X)=j-1.\end{cases}
\end{align*}
We claim that (1) $G^{(j)}$ depends on $\theta$ only through $\theta_{\leq j-1}$. Hence, a run with $G^{(j)}$ is unlikely to activate a coordinate $\ell\geq j$. We claim that (2) $\widetilde{G}_j$ satisfies $\widetilde{G}_j(X,\xi)=G_\theta(X,\xi)$ whenever $\operatorname{prog}(X)\neq j-1$. Hence, by Lemma~\ref{lemma:stochastic-oracle}, item~3, its algorithm history is close to that with $G_\theta$ in relative entropy. We claim that (3) $\widetilde{G}_j(X,\xi)=G^{(j)}(X,\xi)$ whenever $\operatorname{prog}(X)\leq j-1$. This links the two algorithm histories. We verify the claims. (2) holds by definition. (1) holds since $\operatorname{prog}_{<j}$ involves only $\theta_{\leq j-1}$, $G_{\theta,k}$ with $k<j-1$ involves only $\theta_{\leq k+1}\subseteq\theta_{\leq j-1}$, and $Q_{j-1}$ involves only $\theta_{\leq j-1}$ (Lemma~\ref{lemma:stochastic-oracle}). Moreover, (3) holds because if $\operatorname{prog}(X)\leq j-1$, then $\operatorname{prog}_{<j}(X)=\operatorname{prog}(X)$, and hence, for every $\xi$,
\begin{equation}\label{eq:aux-oracles-equal}
\widetilde{G}_j(X,\xi)=G^{(j)}(X,\xi)=\begin{cases}G_{\theta,\operatorname{prog}(X)}(X,\xi),&\operatorname{prog}(X)<j-1,\\ Q_{j-1}^{\theta_{\leq j-1}}(X,\xi),&\operatorname{prog}(X)=j-1.\end{cases}
\end{equation}
We run $\mathsf{A}$ with each of the three oracles $G_\theta$, $\widetilde{G}_j$, and $G^{(j)}$ on the same $(\theta,\omega,\xi_1,\ldots,\xi_T)$, and denote the resulting queries and replies by $(X_t,Y_t)$, $(\widetilde{X}_t,\widetilde{Y}_t)$, and $(X_t^{(j)},Y_t^{(j)})$, respectively. Similarly to $E_j$ in~\eqref{eq:Ej-def}, we define the events
\begin{equation*}
\widetilde{E}_j:=\{\exists t\in[T+1]:\operatorname{prog}(\widetilde{X}_t)\geq j\}, \quad E_j^{(j)}:=\{\exists t\in[T+1]:\operatorname{prog}(X_t^{(j)})\geq j\}.
\end{equation*}
The argument below requires an upper bound on $\mathbb{P}(\widetilde{E}_j)$. We claim that
\begin{equation}
\mathbb{P}(\widetilde{E}_j)=\mathbb{P}(E_j^{(j)})\leq2K(T+1)e^{-mN/128}.\label{eq:aux-bound}
\end{equation}

We first prove $\mathbb{P}(\widetilde{E}_j)=\mathbb{P}(E_j^{(j)})$. We show that $\mathbf{1}_{\widetilde{E}_j}=\mathbf{1}_{E_j^{(j)}}$ everywhere. Let $\tau$ be the first $t\in[T+1]$ with $\operatorname{prog}(\widetilde{X}_t)\geq j$, and $\tau=T+2$ if there is none. We show by induction on $t$ that $\widetilde{X}_t=X_t^{(j)}$ for all $t\leq\min\{\tau,T+1\}$. For $t=1$, both equal $\mathsf{A}_1(\omega)$. If $\widetilde{X}_t=X_t^{(j)}$ for some $t<\tau$, then $\operatorname{prog}(\widetilde{X}_t)\leq j-1$, and~\eqref{eq:aux-oracles-equal} gives $\widetilde{Y}_t=\widetilde{G}_j(\widetilde{X}_t,\xi_t)=G^{(j)}(X_t^{(j)},\xi_t)=Y_t^{(j)}$. Hence $\widetilde{X}_{t+1}=\mathsf{A}_{t+1}(\omega,\widetilde{Y}_1,\ldots,\widetilde{Y}_t)=\mathsf{A}_{t+1}(\omega,Y_1^{(j)},\ldots,Y_t^{(j)})=X_{t+1}^{(j)}$. If $\tau\leq T+1$, then $X_\tau^{(j)}=\widetilde{X}_\tau$ satisfies $\operatorname{prog}(X_\tau^{(j)})\geq j$, and hence both $\widetilde{E}_j$ and $E_j^{(j)}$ occur. If $\tau=T+2$, then $X_t^{(j)}=\widetilde{X}_t$ satisfies $\operatorname{prog}(X_t^{(j)})\leq j-1$ for all $t\in[T+1]$, and hence neither event occurs. Therefore $\mathbf{1}_{\widetilde{E}_j}=\mathbf{1}_{E_j^{(j)}}$, which implies $\mathbb{P}(\widetilde{E}_j)=\mathbb{P}(E_j^{(j)})$.

We then prove $\mathbb{P}(E_j^{(j)})\leq2K(T+1)e^{-mN/128}$. Since $G^{(j)}$ depends on $\theta$ only through $\theta_{\leq j-1}$, induction on $t$ shows that every query $X_t^{(j)}$ with $t\in[T+1]$ is a measurable function of $(\theta_{\leq j-1},\omega,\xi_1,\ldots,\xi_T)$. This tuple is independent of $\theta_\ell$ for every $\ell\geq j$. Applying~\eqref{eq:stochastic-hidden-sign} with $X=X_t^{(j)}$ to each pair $(t,\ell)\in[T+1]\times\{j,\ldots,K\}$, taking expectations, and taking a union bound gives
\begin{equation*}
\mathbb{P}(E_j^{(j)})\leq\textstyle\sum_{t=1}^{T+1}\textstyle\sum_{\ell=j}^K\mathbb{P}(|q_\ell^\theta(X_t^{(j)})|>\tfrac12)\leq2K(T+1)e^{-mN/128}.
\end{equation*}

Let $\mathcal{H}_t:=(\theta,\omega,Y_1,\ldots,Y_t)$ and $\widetilde{\mathcal{H}}_t:=(\theta,\omega,\widetilde{Y}_1,\ldots,\widetilde{Y}_t)$ be the histories of the runs with $G_\theta$ and $\widetilde{G}_j$. Next, we show that an upper bound on $\operatorname{KL}(\mathcal{H}_T\|\widetilde{\mathcal{H}}_T)$ gives an upper bound on $\mathbb{P}(E_j)$. For $h=(\theta,\omega,y_1,\ldots,y_T)$, let $x_t:=\mathsf{A}_t(\omega,y_1,\ldots,y_{t-1})$ for $t\in[T+1]$, and let $\phi(h):=1$ if $\operatorname{prog}(x_t)\geq j$ for some $t\in[T+1]$ and $\phi(h):=0$ otherwise. Then $\mathbf{1}_{E_j}=\phi(\mathcal{H}_T)$ and $\mathbf{1}_{\widetilde{E}_j}=\phi(\widetilde{\mathcal{H}}_T)$. By the data-processing inequality, $\operatorname{kl}(\mathbb{P}(E_j)\|\mathbb{P}(\widetilde{E}_j))\leq\operatorname{KL}(\mathcal{H}_T\|\widetilde{\mathcal{H}}_T)$, where $\operatorname{kl}$ denotes the binary relative entropy. Using $\operatorname{kl}(u\|v)\geq u\log(1/v)-\log2$ and~\eqref{eq:aux-bound}, we obtain
\begin{align*}
\mathbb{P}(E_j)\left(\tfrac{mN}{128}-\log(2K(T+1))\right)-\log2&\leq\mathbb{P}(E_j)\log\tfrac{1}{\mathbb{P}(\widetilde{E}_j)}-\log2\\
&\leq\operatorname{kl}(\mathbb{P}(E_j)\|\mathbb{P}(\widetilde{E}_j))\leq\operatorname{KL}(\mathcal{H}_T\|\widetilde{\mathcal{H}}_T).
\end{align*}
Hence~\eqref{eq:stochastic-stage-cost} follows from
\begin{equation}\label{eq:KL-upper}
\operatorname{KL}(\mathcal{H}_T\|\widetilde{\mathcal{H}}_T)\leq\tfrac{N\nu}{2}\mathbb{E}L_j.
\end{equation}
We now prove~\eqref{eq:KL-upper}. Since $\mathcal{H}_0=\widetilde{\mathcal{H}}_0=(\theta,\omega)$, the chain rule for relative entropy gives
\begin{align*}
\operatorname{KL}(\mathcal{H}_T\|\widetilde{\mathcal{H}}_T)&=\textstyle\sum_{t=1}^T\mathbb{E}\left[\Gamma_t(\mathcal{H}_{t-1})\right],\\
\Gamma_t(h)&:=\operatorname{KL}\left(\mathbb{P}(Y_t\in\cdot\mid\mathcal{H}_{t-1}=h)\middle\|\mathbb{P}(\widetilde{Y}_t\in\cdot\mid\widetilde{\mathcal{H}}_{t-1}=h)\right).
\end{align*}
Fix a history $h=(\theta,\omega,y_1,\ldots,y_{t-1})$ and let $x_t:=\mathsf{A}_t(\omega,y_1,\ldots,y_{t-1})$ be the query it determines. The noise $\xi_t$ is independent of $(\theta,\omega,\xi_1,\ldots,\xi_{t-1})$, of which both $\mathcal{H}_{t-1}$ and $\widetilde{\mathcal{H}}_{t-1}$ are functions. Hence, for every measurable set $S$,
\begin{equation*}
\mathbb{P}(Y_t\in S\mid\mathcal{H}_{t-1}=h)=\mathbb{P}(G_\theta(x_t,\xi)\in S),\quad\mathbb{P}(\widetilde{Y}_t\in S\mid\widetilde{\mathcal{H}}_{t-1}=h)=\mathbb{P}(\widetilde{G}_j(x_t,\xi)\in S),
\end{equation*}
where on the right-hand sides $\theta$ and $x_t$ are fixed and only $\xi$ is random. If $\operatorname{prog}(x_t)\neq j-1$, then $\widetilde{G}_j(x_t,\xi)=G_\theta(x_t,\xi)$, and hence $\Gamma_t(h)=0$. If $\operatorname{prog}(x_t)=j-1$, then $G_\theta(x_t,\xi)=G_{\theta,j-1}(x_t,\xi)$ and $\widetilde{G}_j(x_t,\xi)=Q_{j-1}^{\theta_{\leq j-1}}(x_t,\xi)$, and hence Lemma~\ref{lemma:stochastic-oracle} gives $\Gamma_t(h)\leq\frac{N\nu}{2}$. Therefore $\Gamma_t(\mathcal{H}_{t-1})\leq\frac{N\nu}{2}\mathbf{1}\{\operatorname{prog}(X_t)=j-1\}$, and summing over $t$ gives
\begin{equation*}
\operatorname{KL}(\mathcal{H}_T\|\widetilde{\mathcal{H}}_T)\leq\tfrac{N\nu}{2}\textstyle\sum_{t=1}^T\mathbb{P}(\operatorname{prog}(X_t)=j-1)=\tfrac{N\nu}{2}\mathbb{E}L_j,
\end{equation*}
which is~\eqref{eq:KL-upper}. This completes the proof.
\end{proof}

Finally, we choose the parameters and combine Lemmas~\ref{lemma:stochastic-matrix}--\ref{lemma:stochastic-adaptive-output}.
\begin{proof}\textbf{of Theorem~\ref{thm:main}}
We choose the scales such that the conditions of Lemmas~\ref{lemma:stochastic-matrix} and~\ref{lemma:stochastic-adaptive-output} hold with $K=\Theta(\Delta L/\epsilon^2)$. Let
\begin{equation*}
A:=\tfrac{16C_{s}\epsilon^2}{L},\ \rho:=\tfrac{4C_{s}\epsilon}{L},\ K:=\left\lfloor\tfrac{\Delta}{\Delta_0 A}\right\rfloor,\ T:=\left\lfloor cm\Delta L\sigma_0^{\frac{p}{p-1}}\epsilon^{-\frac{3p-2}{p-1}}\right\rfloor,\ c:=\tfrac{c_\star}{32\Delta_0C_s2000^{p/(p-1)}}.
\end{equation*}
Then $A/\rho=4\epsilon$, $C_sA/\rho^2=L$, and $\Delta_0AK\leq\Delta$. For $c'_p$ small enough, the condition $\epsilon\leq c'_p\min\{\sqrt{\Delta L},\sigma_0\}$ ensures $\epsilon\leq\sigma_0/2000$, $K=\Theta(\Delta L/\epsilon^2)$ with $K\geq1$, and $T\geq1$. Since $\frac{Km}{\nu}\geq\frac{\Delta L}{32\Delta_0C_s\epsilon^2}\cdot m\left(\frac{\sigma_0}{2000\epsilon}\right)^{\frac{p}{p-1}}$, the choice of $c$ gives $T\leq c_\star Km/\nu$.

It remains to verify the dimension condition $mN\geq C_\star\log(2(K+1)(T+1))$. Since $mN\geq\max\{m,\frac{n}{2K}\}$, we have $\log(2(K+1)(T+1))=O(\log(3+\frac{\Delta L}{\epsilon^2})+\frac{p}{p-1}\log(3+\frac{\sigma_0}{\epsilon})+\log m)$. The first two terms are bounded by a constant multiple of $\frac{n}{2K}$ by Eq.~\eqref{eq:dim-lb}, and the third term satisfies $C_\star\log m\leq\frac{m}{2}+C_\star\log(2C_\star)$. Therefore, $mN\geq C_\star\log(2(K+1)(T+1))$ holds once $C_p$ is large enough. In addition, $K\leq n$ holds by Eq.~\eqref{eq:dim-lb}.

We draw $\theta$ uniformly. By Lemmas~\ref{lemma:stochastic-matrix} and~\ref{lemma:stochastic-oracle}, $F_{\theta,K}\in\mathcal{F}_\textnormal{op}(m,n,\Delta,L)$ and $G_\theta$ is a valid oracle in $\mathcal{O}_p(\sigma_0)$. Outside the family $\{F_{\theta,K}\}$, the oracle returns exact gradients. By Lemma~\ref{lemma:stochastic-adaptive-output}, for every randomized algorithm and every $s\in[T+1]$, its query point $X_s$ satisfies
\begin{equation*}
\mathbb{E}_{\theta,\omega,\xi_1,\ldots,\xi_T}\|\nabla F_{\theta,K}(X_s)\|_\textnormal{nuc}\geq2\epsilon\cdot\tfrac34>\epsilon.
\end{equation*}
Hence no query index $s\leq T+1$ satisfies the stationarity requirement in~\eqref{eq:def-complexity}, and $\mathfrak{m}^{\mathrm{rand}}_{\epsilon,p}(m,n,\Delta,L,\sigma_0)>T\geq\frac{c}{2}m\Delta L\sigma_0^{\frac{p}{p-1}}\epsilon^{-\frac{3p-2}{p-1}}$, completing the proof.
\end{proof}
\subsection{Smooth and nonconvex problems}
We prove the descent lemma.
\begin{proposition}
\label{prop:generalized-smooth}
Under Assumption~\ref{assumption:smooth}, for any $X,Y\in\mathbb{R}^{m\times n}$ satisfying $\|Y-X\|\le \frac{1}{L_1}$, we have
\begin{equation*}
F(Y)\le F(X)+\langle\nabla F(X),Y-X\rangle+\tfrac{L_0+L_1\|\nabla F(X)\|_\star}{2}\|Y-X\|^2.
\end{equation*}
\end{proposition}
\begin{proof}
For every $u\in[0,1]$, Assumption~\ref{assumption:smooth} gives
\begin{equation*}
\begin{aligned}
F(Y)-F(X)-\langle\nabla F(X),Y-X\rangle&=\int_0^1\langle\nabla F(X+u(Y-X))-\nabla F(X),Y-X\rangle du\\
&\le\bigl(L_0+L_1\|\nabla F(X)\|_\star\bigr)\|Y-X\|^2\int_0^1udu,
\end{aligned}
\end{equation*}
which proves the claim.
\end{proof}
We show a direct consequence of Eq.~\eqref{def:factor}, bounding heavy-tailed noises.
\begin{lemma}
\label{lem:weighted-batched-noise}
Let $p\in(1,2]$, $t,B\in\mathbb{N}$, and $\tau_\star=\tau(\|\cdot\|_\star,m,n,p)$. Let $(\mathcal H_s)_{s=0}^t$ be a filtration. For each $s=0,\ldots,t-1$, let $\Gamma_s\geq0$ be an $\mathcal H_s$-measurable random variable, and let $\zeta_s^1,\ldots,\zeta_s^B$ be $\mathcal H_{s+1}$-measurable and conditionally independent given $\mathcal H_s$. Suppose that $\mathbb E[\zeta_s^i\mid\mathcal H_s]=0$ and $\mathbb E[\|\zeta_s^i\|_\star^p\mid\mathcal H_s]\le\sigma_0^p+\sigma_1^p\Gamma_s^p$ for every $s<t$ and $i\in[B]$. Define $\bar\zeta_s:=B^{-1}\sum_{i=1}^B\zeta_s^i$. Then, for any deterministic $a_0,\ldots,a_{t-1}\ge0$, we have
\begin{equation*}
\mathbb E\left\|\textstyle\sum_{s=0}^{t-1}a_s\bar\zeta_s\right\|_\star\le\tfrac{\tau_\star}{B^{\frac{p-1}{p}}}\left[\sigma_0\left(\textstyle\sum_{s=0}^{t-1}a_s^p\right)^{\frac1p}+\sigma_1\textstyle\sum_{s=0}^{t-1}a_s\mathbb E\Gamma_s\right].
\end{equation*}
\end{lemma}

\begin{proof}
Conditional independence makes the lexicographically ordered $\{a_s\zeta_s^i\}_{s,i}$ a martingale difference sequence. Hence, by Eq.~\eqref{def:factor}, we have
\begin{equation}
\label{eq:sum-as-zetas}
\mathbb E\left\|\textstyle\sum_{s=0}^{t-1}a_s\bar\zeta_s\right\|_\star\le\tfrac{\tau_\star}{B}\mathbb E\left(\textstyle\sum_{s=0}^{t-1}\sum_{i=1}^B a_s^p\|\zeta_s^i\|_\star^p\right)^{\frac1p}.
\end{equation}
For any $k\in\{0,\ldots,t-1\}$ and any nonnegative $\mathcal H_k$-measurable random variable $U$, we have
\begin{align*}
\mathbb E\left[\left(U+a_k^p\textstyle\sum_{i=1}^B\|\zeta_k^i\|_\star^p\right)^{\frac1p}\middle|\mathcal H_k\right]
& \leq \left(U+a_k^p\textstyle\sum_{i=1}^B\mathbb E\left[\|\zeta_k^i\|_\star^p\mid\mathcal H_k\right]\right)^{\frac1p} \\
& \leq  \left(U+Ba_k^p\sigma_0^p+Ba_k^p\sigma_1^p\Gamma_k^p\right)^{\frac1p} \\
& \leq \left(U+Ba_k^p\sigma_0^p\right)^{\frac1p}+B^{\frac1p}a_k\sigma_1\Gamma_k.
\end{align*}
Here, the first inequality follows from conditional Jensen's inequality, the second inequality uses $\mathbb E[\|\zeta_k^i\|_\star^p\mid\mathcal H_k]\le\sigma_0^p+\sigma_1^p\Gamma_k^p$ for every $i\in[B]$, and the third inequality uses $(a+b)^{\frac1p}\le a^{\frac1p}+b^{\frac1p}$ for $a,b\geq0$.
Applying the above inequality successively for $k=t-1,\ldots,0$, together with the tower property, yields
\begin{align*}
\mathbb E\left(\textstyle\sum_{s=0}^{t-1}\sum_{i=1}^B a_s^p\|\zeta_s^i\|_\star^p\right)^{\frac1p}
& \leq \mathbb E\left(\textstyle\sum_{s=0}^{t-2}\sum_{i=1}^B a_s^p\|\zeta_s^i\|_\star^p+B\sigma_0^pa_{t-1}^p\right)^{\frac1p}+B^{\frac1p}\sigma_1a_{t-1}\mathbb E\Gamma_{t-1} \\
& \leq \mathbb E\left(\textstyle\sum_{s=0}^{k-1}\sum_{i=1}^B a_s^p\|\zeta_s^i\|_\star^p+B\sigma_0^p\sum_{s=k}^{t-1}a_s^p\right)^{\frac1p}+B^{\frac1p}\sigma_1\textstyle\sum_{s=k}^{t-1}a_s\mathbb E\Gamma_s \\
& \leq \left(B\sigma_0^p\textstyle\sum_{s=0}^{t-1}a_s^p\right)^{\frac1p}+B^{\frac1p}\sigma_1\textstyle\sum_{s=0}^{t-1}a_s\mathbb E\Gamma_s,
\end{align*}
where the second line holds for every $k\in\{0,\ldots,t-1\}$ by induction. Combining with Eq.~\eqref{eq:sum-as-zetas} gives the desired result.
\end{proof}

\begin{proof}\textbf{of Theorem~\ref{thm:nonconvex-smooth}}
Since $\widetilde X_T$ is uniformly chosen from $\{X_0,\ldots,X_{T-1}\}$, we have
\begin{equation*}
\mathbb E\|\nabla F(\widetilde X_T)\|_\star=\tfrac{1}{T}\textstyle\sum_{t=0}^{T-1}\mathbb E\|\nabla F(X_t)\|_\star.
\end{equation*}
The choices of $B$ and $\eta$ give
\begin{equation}\label{eq:hb-param}
\tfrac{\tau_\star\sigma_1}{B^{\frac{p-1}{p}}}\leq\tfrac{1-\delta}{16}, \quad L_1\eta\leq\tfrac{(1-\delta)\alpha}{8}\leq\tfrac{1}{8}.
\end{equation}
In particular, $\|X_{t+1}-X_t\|=\eta\|\operatorname{lmo}_\delta(m_{t+1})\|\leq\eta\leq\frac{1}{L_1}$ for all $t\geq 0$. By applying Proposition~\ref{prop:generalized-smooth} with $X=X_t$ and $Y=X_{t+1}$, we have
\begin{equation}\label{eq:hb-descent}
F(X_{t+1})-F(X_t)\leq\eta\langle \nabla F(X_t),\operatorname{lmo}_\delta(m_{t+1})\rangle+\tfrac{\eta^2}{2}\left(L_0+L_1\|\nabla F(X_t)\|_\star\right).
\end{equation}
We write $m_{t+1}=\nabla F(X_t)+E_t+N_t$ where
\begin{equation*}
\begin{array}{lcl}
E_t & = & \sum_{s=1}^t\beta^{t-s+1}(\nabla F(X_{s-1})-\nabla F(X_s)), \\
N_t & = & \beta^t(\bar G_0-\nabla F(X_0))+\sum_{s=1}^t\alpha\beta^{t-s}(\bar G_s-\nabla F(X_s)).
\end{array}
\end{equation*}
By Assumption~\ref{assumption:inexactness}, we have
\begin{equation*}
\begin{array}{rcl}
\langle \nabla F(X_t),\operatorname{lmo}_\delta(m_{t+1})\rangle & = & \langle m_{t+1},\operatorname{lmo}_\delta(m_{t+1})\rangle-\langle E_t+N_t,\operatorname{lmo}_\delta(m_{t+1})\rangle \\
& \leq & -(1-\delta)\|m_{t+1}\|_\star+\|E_t+N_t\|_\star \\
& \leq & -(1-\delta)\|\nabla F(X_t)\|_\star+2\|E_t\|_\star+2\|N_t\|_\star.
\end{array}
\end{equation*}
Combining this inequality with Eq.~\eqref{eq:hb-descent} yields
\begin{equation*}
F(X_{t+1})-F(X_t)\leq-\eta\left(1-\delta-\tfrac{L_1\eta}{2}\right)\|\nabla F(X_t)\|_\star+2\eta\|E_t\|_\star+2\eta\|N_t\|_\star+\tfrac{L_0\eta^2}{2}.
\end{equation*}
Summing the above inequality over $t=0,\ldots,T-1$, using $F(X_T)\geq F^\star$, dividing by $\eta T$, and taking expectations yields
\begin{equation}\label{eq:hb-summed-descent}
\left(1-\delta-\tfrac{L_1\eta}{2}\right)\mathbb E\|\nabla F(\widetilde X_T)\|_\star\leq\tfrac{\Delta_0}{\eta T}+\tfrac{L_0\eta}{2}+\tfrac{2}{T}\textstyle\sum_{t=0}^{T-1}\mathbb E\|E_t\|_\star+\tfrac{2}{T}\textstyle\sum_{t=0}^{T-1}\mathbb E\|N_t\|_\star.
\end{equation}

First, we bound $\sum_{t=0}^{T-1}\mathbb E\|E_t\|_\star$. Indeed, since $\|X_s-X_{s-1}\|\leq\eta\leq\frac{1}{L_1}$, Assumption~\ref{assumption:smooth} gives $\|\nabla F(X_{s-1})-\nabla F(X_s)\|_\star\leq\eta(L_0+L_1\|\nabla F(X_{s-1})\|_\star)$ for $s\geq 1$. This implies
\begin{equation*}
\|E_t\|_\star\leq\eta\textstyle\sum_{s=1}^t\beta^{t-s+1}\left(L_0+L_1\|\nabla F(X_{s-1})\|_\star\right)\leq\tfrac{L_0\eta}{\alpha}+L_1\eta\textstyle\sum_{s=0}^{t-1}\beta^{t-s}\|\nabla F(X_s)\|_\star.
\end{equation*}
Exchanging the order of summation and using $\sum_{t=s+1}^{T-1}\beta^{t-s}\leq\frac{1}{\alpha}$, we have
\begin{equation*}
\tfrac{1}{T}\textstyle\sum_{t=0}^{T-1}\mathbb E\|E_t\|_\star\leq\tfrac{L_0\eta}{\alpha}+\tfrac{L_1\eta}{\alpha}\mathbb E\|\nabla F(\widetilde X_T)\|_\star.
\end{equation*}

Second, we bound $\sum_{t=0}^{T-1}\mathbb E\|N_t\|_\star$. Indeed, we let $\mathcal H_s$ be the sigma algebra generated by $X_0$ and $\{\xi_j^i\}_{j<s,i\in[B]}$, and let $\zeta_s^i:=G(X_s,\xi_s^i)-\nabla F(X_s)$. Then $X_s$ is $\mathcal H_s$-measurable, $\zeta_s^i$ is $\mathcal H_{s+1}$-measurable, and Assumption~\ref{assumption:noise} and the fresh i.i.d.\ mini-batch at $X_s$ give that $\zeta_s^1,\ldots,\zeta_s^B$ are conditionally independent given $\mathcal H_s$ with
\begin{equation*}
\mathbb E[\zeta_s^i\mid\mathcal H_s]=0, \quad \mathbb E[\|\zeta_s^i\|_\star^p\mid\mathcal H_s]\leq\sigma_0^p+\sigma_1^p\|\nabla F(X_s)\|_\star^p.
\end{equation*}
We treat the initial batch separately from the later ones, since applying Lemma~\ref{lem:weighted-batched-noise} to it would introduce the dimension-dependent factor $\tau_\star$ into the initialization term. Since $\bar G_0-\nabla F(X_0)=\frac{1}{B}\sum_{i=1}^B\zeta_0^i$, the triangle inequality, Jensen's inequality, and $(a+b)^{\frac1p}\leq a^{\frac1p}+b^{\frac1p}$ give
\begin{equation*}
\mathbb E\|\bar G_0-\nabla F(X_0)\|_\star\leq\tfrac1B\textstyle\sum_{i=1}^B\left(\mathbb E\|\zeta_0^i\|_\star^p\right)^{\frac1p}\leq\sigma_0+\sigma_1\|\nabla F(X_0)\|_\star.
\end{equation*}
Since $\bar G_s-\nabla F(X_s)=\frac{1}{B}\sum_{i=1}^B\zeta_s^i$, applying Lemma~\ref{lem:weighted-batched-noise} to the $t$ batches with indices $1\leq s\leq t$, $\Gamma_s=\|\nabla F(X_s)\|_\star$, and weights $a_s=\alpha\beta^{t-s}$ gives
\begin{equation*}
\mathbb E\|N_t\|_\star\leq\beta^t(\sigma_0+\sigma_1\|\nabla F(X_0)\|_\star)+\tfrac{\tau_\star}{B^{\frac{p-1}{p}}}\left[\sigma_0\left(\alpha^p\textstyle\sum_{s=1}^t\beta^{p(t-s)}\right)^{\frac1p}+\sigma_1\textstyle\sum_{s=1}^t\alpha\beta^{t-s}\mathbb E\|\nabla F(X_s)\|_\star\right].
\end{equation*}
Since $1-\beta^p\geq1-\beta=\alpha$, we have $\alpha(\sum_{j\geq0}\beta^{pj})^{\frac1p}=\alpha(1-\beta^p)^{-\frac1p}\leq\alpha^{\frac{p-1}{p}}$, and this implies
\begin{equation*}
\mathbb E\|N_t\|_\star\leq\beta^t(\sigma_0+\sigma_1\|\nabla F(X_0)\|_\star)+\tfrac{\tau_\star}{B^{\frac{p-1}{p}}}\left[\sigma_0\alpha^{\frac{p-1}{p}}+\sigma_1\textstyle\sum_{s=1}^t\alpha\beta^{t-s}\mathbb E\|\nabla F(X_s)\|_\star\right].
\end{equation*}
Using $\sum_{t=0}^{T-1}\beta^t\leq\frac1\alpha$ and $\sum_{t=s}^{T-1}\alpha\beta^{t-s}\leq1$, we have
\begin{equation*}
\tfrac{1}{T}\textstyle\sum_{t=0}^{T-1}\mathbb E\|N_t\|_\star\leq\tfrac{\sigma_0+\sigma_1\|\nabla F(X_0)\|_\star}{\alpha T}+\tfrac{\tau_\star\sigma_0\alpha^{\frac{p-1}{p}}}{B^{\frac{p-1}{p}}}+\tfrac{\tau_\star\sigma_1}{B^{\frac{p-1}{p}}}\mathbb E\|\nabla F(\widetilde X_T)\|_\star.
\end{equation*}

Putting these pieces together with Eq.~\eqref{eq:hb-summed-descent}, we have
\begin{equation}\label{eq:hb-combined}
\left(1-\delta-\tfrac{L_1\eta}{2}-\tfrac{2L_1\eta}{\alpha}-\tfrac{2\tau_\star\sigma_1}{B^{\frac{p-1}{p}}}\right)\mathbb E\|\nabla F(\widetilde X_T)\|_\star\leq\tfrac{\Delta_0}{\eta T}+\tfrac{L_0\eta}{2}+\tfrac{2L_0\eta}{\alpha}+\tfrac{2(\sigma_0+\sigma_1\|\nabla F(X_0)\|_\star)}{\alpha T}+\tfrac{2\tau_\star\sigma_0\alpha^{\frac{p-1}{p}}}{B^{\frac{p-1}{p}}}.
\end{equation}
By Eq.~\eqref{eq:hb-param}, the coefficient on the left-hand side is at least $(1-\delta)(1-\frac{1}{16}-\frac{1}{4}-\frac{1}{8})=\frac{9(1-\delta)}{16}$. By the definition of $\eta$, we have
\begin{equation*}
\tfrac{1}{\eta}\leq\sqrt{\tfrac{L_0T}{\alpha\Delta_0}}+\tfrac{8L_1}{(1-\delta)\alpha}, \quad \tfrac{L_0\eta}{\alpha}\leq\sqrt{\tfrac{L_0\Delta_0}{\alpha T}}.
\end{equation*}
This implies $\frac{\Delta_0}{\eta T}\leq\sqrt{\frac{L_0\Delta_0}{\alpha T}}+\frac{8L_1\Delta_0}{(1-\delta)\alpha T}$ and $\frac{L_0\eta}{2}+\frac{2L_0\eta}{\alpha}\leq\frac52\sqrt{\frac{L_0\Delta_0}{\alpha T}}$. Thus, Eq.~\eqref{eq:hb-combined} gives
\begin{equation}\label{eq:hb-master}
\mathbb E\|\nabla F(\widetilde X_T)\|_\star\leq\tfrac{128}{9(1-\delta)}\left[\tfrac{1}{\alpha T}\left(\tfrac{L_1\Delta_0}{1-\delta}+\sigma_0+\sigma_1\|\nabla F(X_0)\|_\star\right)+\sqrt{\tfrac{L_0\Delta_0}{\alpha T}}+\tfrac{\tau_\star\sigma_0\alpha^{\frac{p-1}{p}}}{B^{\frac{p-1}{p}}}\right].
\end{equation}
We choose
\begin{equation*}
T:=\left\lceil\max\left\{1,\tfrac{120}{(1-\delta)\alpha\epsilon}\left(\tfrac{L_1\Delta_0}{1-\delta}+\sigma_0+\sigma_1\|\nabla F(X_0)\|_\star\right),\tfrac{L_0\Delta_0}{\alpha}\left(\tfrac{120}{(1-\delta)\epsilon}\right)^2\right\}\right\rceil.
\end{equation*}
The first two terms in the brackets of Eq.~\eqref{eq:hb-master} are at most $\frac{(1-\delta)\epsilon}{120}$. By the definition of $\alpha$, we have $\frac{\tau_\star\sigma_0\alpha^{\frac{p-1}{p}}}{B^{\frac{p-1}{p}}}\leq\frac{(1-\delta)\epsilon}{120}$. Therefore, we have
\begin{equation*}
\mathbb E\|\nabla F(\widetilde X_T)\|_\star\leq\tfrac{128}{9(1-\delta)}\cdot\tfrac{3(1-\delta)\epsilon}{120}=\tfrac{16\epsilon}{45}<\epsilon.
\end{equation*}
It remains to count oracle calls. Since Algorithm~\ref{algorithm:Muon-HB} calls the stochastic gradient oracle $B$ times per iteration, the total number of calls is $BT$. Let $q:=\frac{p}{p-1}$ and $y_\epsilon:=(\frac{120\tau_\star\sigma_0}{(1-\delta)\epsilon})^q$. By the definition of $\alpha$, we have $\frac{B}{\alpha}=\max\{B,y_\epsilon\}$. Then, we have
\begin{equation*}
BT\leq2B+\tfrac{B}{\alpha}\left(S_\epsilon+D_\epsilon\right),\quad S_\epsilon:=\tfrac{120}{(1-\delta)\epsilon}\left(\tfrac{L_1\Delta_0}{1-\delta}+\sigma_0+\sigma_1\|\nabla F(X_0)\|_\star\right),\quad D_\epsilon:=\tfrac{120^2L_0\Delta_0}{(1-\delta)^2\epsilon^2}.
\end{equation*}
We show that the product $y_\epsilon D_\epsilon$ dominates for all $\epsilon\in(0,\epsilon_0]$, where
\begin{equation*}
\epsilon_0:=\min\left\{\sigma_0,\tfrac{3\sigma_0}{\sigma_1},\sqrt{L_0\Delta_0},\tfrac{L_0\Delta_0}{L_1\Delta_0+\sigma_0+\sigma_1\|\nabla F(X_0)\|_\star}\right\}
\end{equation*}
and $\frac{3\sigma_0}{\sigma_1}:=+\infty$ if $\sigma_1=0$. Note that $\epsilon_0$ does not depend on $m$, $n$, or $\tau_\star$. First, $\tau_\star\geq1$ by Lemma~\ref{lemma:dimension-upper-bound} and $q\geq2$, and hence $B\leq1+(\frac{16\tau_\star\sigma_1}{1-\delta})^q\leq\tau_\star^q(2\max\{1,\frac{16\sigma_1}{1-\delta}\})^q\leq y_\epsilon$. Thus, $\frac{B}{\alpha}=y_\epsilon$. Second, $\epsilon\leq\frac{L_0\Delta_0}{L_1\Delta_0+\sigma_0+\sigma_1\|\nabla F(X_0)\|_\star}$ gives $S_\epsilon\leq D_\epsilon$. Third, $\epsilon\leq\sqrt{L_0\Delta_0}$ gives $D_\epsilon\geq1$, and hence $2B\leq2y_\epsilon D_\epsilon$. Therefore, we have
\begin{equation*}
BT\leq4y_\epsilon D_\epsilon=4\cdot120^{\frac{3p-2}{p-1}}\tau_\star^{\frac{p}{p-1}}\sigma_0^{\frac{p}{p-1}}L_0\Delta_0(1-\delta)^{-\frac{3p-2}{p-1}}\epsilon^{-\frac{3p-2}{p-1}}.
\end{equation*}
For $\|\cdot\|=\|\cdot\|_\textnormal{op}$ and $\|\cdot\|_\star=\|\cdot\|_\textnormal{nuc}$, Lemma~\ref{lemma:dimension-upper-bound} gives $\tau_\star^{\frac{p}{p-1}}=\Theta(\min\{m,n\})$, which proves the desired result. 
\end{proof}

\subsection{Highly smooth and nonconvex problems}
\begin{proof}
\textbf{of Theorem~\ref{thm:transported-uSCG}}
Since $\widetilde X_T$ is uniformly chosen from $\{X_0,\ldots,X_{T-1}\}$, we have
\begin{equation*}
\mathbb E\|\nabla F(\widetilde X_T)\|_\star=\tfrac{1}{T}\textstyle\sum_{t=0}^{T-1}\mathbb E\|\nabla F(X_t)\|_\star.
\end{equation*}
The choices of $B$, $\alpha$, and $\eta$ give
\begin{equation}\label{eq:tr-param}
\begin{array}{c}
\tfrac{\tau_\star\sigma_1}{B^{\frac{p-1}{p}}}\leq\tfrac{1-\delta}{8}, \quad \tfrac{\tau_\star\sigma_1L_0}{B^{\frac{p-1}{p}}}\leq\tfrac{\sqrt{(1-\delta)L_2\epsilon}}{4}, \quad \tfrac{\alpha^{\frac{p-1}{p}}}{B^{\frac{p-1}{p}}}\leq\tfrac{(1-\delta)\epsilon}{32\tau_\star\sigma_0}, \\[4pt]
L_0\alpha\leq\sqrt{(1-\delta)L_2\epsilon}, \quad \tfrac{L_1\eta}{\alpha}\leq\tfrac{1-\delta}{8}, \quad \tfrac{\eta}{\alpha}\leq\tfrac18\sqrt{\tfrac{(1-\delta)\epsilon}{L_2}}.
\end{array}
\end{equation}
Since $\|\operatorname{lmo}_\delta(\cdot)\|\leq1$ and $Y_t-X_t=\frac{\beta}{\alpha}(X_t-X_{t-1})$, we have
\begin{equation}\label{eq:tr-steps}
\|X_{t+1}-X_t\|\leq\eta\leq\tfrac{1}{L_1}, \quad \|Y_t-X_t\|\leq\tfrac{\beta\eta}{\alpha}\leq\tfrac{1}{L_1}.
\end{equation}
By applying Proposition~\ref{prop:generalized-smooth} with $X=X_t$ and $Y=X_{t+1}$, we have
\begin{equation}\label{eq:tr-descent}
F(X_{t+1})-F(X_t)\leq\eta\langle \nabla F(X_t),\operatorname{lmo}_\delta(m_{t+1})\rangle+\tfrac{\eta^2}{2}\left(L_0+L_1\|\nabla F(X_t)\|_\star\right).
\end{equation}
Let $\bar\zeta_0:=\bar G_0-\nabla F(X_0)$, $\bar\zeta_t:=\bar G_t-\nabla F(Y_t)$ for $t\geq1$, and
\begin{equation*}
Z(Y,X):=\nabla F(Y)-\nabla F(X)-\nabla^2F(X)[Y-X].
\end{equation*}
For $t\geq1$, we have
\begin{equation*}
m_{t+1}-\nabla F(X_t)=\beta(m_t-\nabla F(X_{t-1}))+\beta(\nabla F(X_{t-1})-\nabla F(X_t))+\alpha(\nabla F(Y_t)-\nabla F(X_t))+\alpha\bar\zeta_t.
\end{equation*}
Since $\alpha(Y_t-X_t)=\beta(X_t-X_{t-1})$, the terms $\beta\nabla^2F(X_t)[X_{t-1}-X_t]$ and $\alpha\nabla^2F(X_t)[Y_t-X_t]$ cancel, and thus
\begin{equation*}
m_{t+1}-\nabla F(X_t)=\beta(m_t-\nabla F(X_{t-1}))+\beta Z(X_{t-1},X_t)+\alpha Z(Y_t,X_t)+\alpha\bar\zeta_t.
\end{equation*}
Since $m_1-\nabla F(X_0)=\bar\zeta_0$, unrolling this recursion yields $m_{t+1}=\nabla F(X_t)+E_t+N_t$, where
\begin{equation*}
\begin{array}{lcl}
E_t & = & \sum_{s=1}^t\beta^{t-s}\left[\beta Z(X_{s-1},X_s)+\alpha Z(Y_s,X_s)\right], \\
N_t & = & \beta^t\bar\zeta_0+\sum_{s=1}^t\alpha\beta^{t-s}\bar\zeta_s.
\end{array}
\end{equation*}
By Assumption~\ref{assumption:inexactness} and $\|m_{t+1}\|_\star\geq\|\nabla F(X_t)\|_\star-\|E_t+N_t\|_\star$, we have
\begin{equation*}
\langle \nabla F(X_t),\operatorname{lmo}_\delta(m_{t+1})\rangle\leq-(1-\delta)\|m_{t+1}\|_\star+\|E_t+N_t\|_\star\leq-(1-\delta)\|\nabla F(X_t)\|_\star+2\|E_t\|_\star+2\|N_t\|_\star.
\end{equation*}
Combining this inequality with Eq.~\eqref{eq:tr-descent}, we have
\begin{equation*}
F(X_{t+1})-F(X_t)\leq-\left(1-\delta-\tfrac{L_1\eta}{2}\right)\eta\|\nabla F(X_t)\|_\star+\tfrac{L_0\eta^2}{2}+2\eta\|E_t\|_\star+2\eta\|N_t\|_\star.
\end{equation*}
Summing over $t=0,\ldots,T-1$, using $F(X_T)\geq F^\star$, dividing by $\eta T$, and taking expectations yields
\begin{equation}\label{eq:tr-summed-descent}
\left(1-\delta-\tfrac{L_1\eta}{2}\right)\mathbb E\|\nabla F(\widetilde X_T)\|_\star\leq\tfrac{\Delta_0}{\eta T}+\tfrac{L_0\eta}{2}+\tfrac{2}{T}\textstyle\sum_{t=0}^{T-1}\mathbb E\|E_t\|_\star+\tfrac{2}{T}\textstyle\sum_{t=0}^{T-1}\mathbb E\|N_t\|_\star.
\end{equation}

First, we bound $\sum_{t=0}^{T-1}\|E_t\|_\star$. Indeed, by Eq.~\eqref{eq:tr-steps}, Assumption~\ref{assumption:hesslip} applies to both pairs $(X_{s-1},X_s)$ and $(Y_s,X_s)$, and gives
\begin{equation*}
\|Z(X_{s-1},X_s)\|_\star\leq L_2\eta^2, \quad \|Z(Y_s,X_s)\|_\star\leq\tfrac{L_2\beta^2\eta^2}{\alpha^2}.
\end{equation*}
This implies $\|\beta Z(X_{s-1},X_s)+\alpha Z(Y_s,X_s)\|_\star\leq\beta L_2\eta^2+\frac{\beta^2L_2\eta^2}{\alpha}\leq\frac{L_2\eta^2}{\alpha}$. Using $\sum_{s=1}^t\beta^{t-s}\leq\frac1\alpha$, we have $\|E_t\|_\star\leq\tfrac{L_2\eta^2}{\alpha^2}$ and hence
\begin{equation*}
\tfrac{1}{T}\textstyle\sum_{t=0}^{T-1}\|E_t\|_\star\leq\tfrac{L_2\eta^2}{\alpha^2}.
\end{equation*}

Second, we bound $\sum_{t=0}^{T-1}\mathbb E\|N_t\|_\star$. Indeed, we let $\mathcal H_s$ be the sigma algebra generated by $X_0$ and $\{\xi_j^i\}_{j<s,i\in[B]}$. We let $\zeta_s^i:=G(Y_s,\xi_s^i)-\nabla F(Y_s)$. Then $Y_s$ is $\mathcal H_s$-measurable, $\zeta_s^i$ is $\mathcal H_{s+1}$-measurable, and Assumption~\ref{assumption:noise} and the fresh i.i.d.\ mini-batch at $Y_s$ give that $\zeta_s^1,\ldots,\zeta_s^B$ are conditionally independent given $\mathcal H_s$ with
\begin{equation*}
\mathbb E[\zeta_s^i\mid\mathcal H_s]=0, \quad \mathbb E[\|\zeta_s^i\|_\star^p\mid\mathcal H_s]\leq\sigma_0^p+\sigma_1^p\|\nabla F(Y_s)\|_\star^p.
\end{equation*}
As in the proof of Theorem~\ref{thm:nonconvex-smooth}, we treat the initial batch separately, and the triangle inequality, Jensen's inequality, and $(a+b)^{\frac1p}\leq a^{\frac1p}+b^{\frac1p}$ give $\mathbb E\|\bar\zeta_0\|_\star\leq\sigma_0+\sigma_1\|\nabla F(X_0)\|_\star$. Since $\bar\zeta_s=\frac1B\sum_{i=1}^B\zeta_s^i$, applying Lemma~\ref{lem:weighted-batched-noise} to the $t$ batches with indices $1\leq s\leq t$, $\Gamma_s=\|\nabla F(Y_s)\|_\star$, and weights $a_s=\alpha\beta^{t-s}$, and using $\alpha(\sum_{j\geq0}\beta^{pj})^{\frac1p}=\alpha(1-\beta^p)^{-\frac1p}\leq\alpha^{\frac{p-1}{p}}$, we have
\begin{equation*}
\mathbb E\|N_t\|_\star\leq\beta^t(\sigma_0+\sigma_1\|\nabla F(X_0)\|_\star)+\tfrac{\tau_\star}{B^{\frac{p-1}{p}}}\left[\sigma_0\alpha^{\frac{p-1}{p}}+\sigma_1\textstyle\sum_{s=1}^t\alpha\beta^{t-s}\mathbb E\|\nabla F(Y_s)\|_\star\right].
\end{equation*}
It remains to compare $\nabla F(Y_s)$ with $\nabla F(X_s)$. By Eq.~\eqref{eq:tr-steps} and Assumption~\ref{assumption:smooth}, we have
\begin{equation*}
\|\nabla F(Y_s)\|_\star\leq\|\nabla F(X_s)\|_\star+\left(L_0+L_1\|\nabla F(X_s)\|_\star\right)\|Y_s-X_s\|\leq\left(1+\tfrac{L_1\eta}{\alpha}\right)\|\nabla F(X_s)\|_\star+\tfrac{L_0\eta}{\alpha}.
\end{equation*}
Using $\sum_{t=0}^{T-1}\beta^t\leq\frac1\alpha$, $\sum_{s=1}^t\alpha\beta^{t-s}\leq1$, and $\sum_{t=s}^{T-1}\alpha\beta^{t-s}\leq1$, we have
\begin{equation*}
\tfrac{1}{T}\textstyle\sum_{t=0}^{T-1}\mathbb E\|N_t\|_\star\leq\tfrac{\sigma_0+\sigma_1\|\nabla F(X_0)\|_\star}{\alpha T}+\tfrac{\tau_\star\sigma_0\alpha^{\frac{p-1}{p}}}{B^{\frac{p-1}{p}}}+\tfrac{\tau_\star\sigma_1L_0\eta}{B^{\frac{p-1}{p}}\alpha}+\tfrac{\tau_\star\sigma_1}{B^{\frac{p-1}{p}}}\left(1+\tfrac{L_1\eta}{\alpha}\right)\mathbb E\|\nabla F(\widetilde X_T)\|_\star.
\end{equation*}

Putting these pieces together with Eq.~\eqref{eq:tr-summed-descent}, we have
\begin{equation}\label{eq:tr-combined}
\begin{array}{rcl}
\lefteqn{\left[1-\delta-\tfrac{L_1\eta}{2}-\tfrac{2\tau_\star\sigma_1}{B^{\frac{p-1}{p}}}\left(1+\tfrac{L_1\eta}{\alpha}\right)\right]\mathbb E\|\nabla F(\widetilde X_T)\|_\star} \\
& \leq & \tfrac{\Delta_0}{\eta T}+\tfrac{L_0\eta}{2}+\tfrac{2(\sigma_0+\sigma_1\|\nabla F(X_0)\|_\star)}{\alpha T}+\tfrac{2\tau_\star\sigma_0\alpha^{\frac{p-1}{p}}}{B^{\frac{p-1}{p}}}+\tfrac{2\tau_\star\sigma_1L_0\eta}{B^{\frac{p-1}{p}}\alpha}+\tfrac{2L_2\eta^2}{\alpha^2}.
\end{array}
\end{equation}
By Eq.~\eqref{eq:tr-param}, we have $\frac{L_1\eta}{2}\leq\frac{1-\delta}{16}$ and $\frac{2\tau_\star\sigma_1}{B^{\frac{p-1}{p}}}(1+\frac{L_1\eta}{\alpha})\leq\frac{1-\delta}{4}\cdot\frac98$. Thus, the coefficient on the left-hand side of Eq.~\eqref{eq:tr-combined} is at least $(1-\delta)(1-\frac{2}{32}-\frac{9}{32})=\frac{21(1-\delta)}{32}$. We choose
\begin{equation*}
T:=\left\lceil\max\left\{1,\tfrac{16\Delta_0}{(1-\delta)\eta\epsilon},\tfrac{32(\sigma_0+\sigma_1\|\nabla F(X_0)\|_\star)}{(1-\delta)\alpha\epsilon}\right\}\right\rceil.
\end{equation*}
We claim that each of the six terms on the right-hand side of Eq.~\eqref{eq:tr-combined} is at most $\frac{(1-\delta)\epsilon}{16}$. Indeed, $\frac{\Delta_0}{\eta T}\leq \frac{(1-\delta)\epsilon}{16}$ and $\frac{2(\sigma_0+\sigma_1\|\nabla F(X_0)\|_\star)}{\alpha T}\leq \frac{(1-\delta)\epsilon}{16}$ follow from the definition of $T$, and the remaining four bounds follow from Eq.~\eqref{eq:tr-param}:
\begin{equation*}
\begin{array}{rcl}
\tfrac{L_0\eta}{2} & = & \tfrac{L_0\alpha}{2}\cdot\tfrac{\eta}{\alpha}\leq\tfrac{\sqrt{(1-\delta)L_2\epsilon}}{2}\cdot\tfrac18\sqrt{\tfrac{(1-\delta)\epsilon}{L_2}}=\tfrac{(1-\delta)\epsilon}{16}, \\[4pt]
\tfrac{2\tau_\star\sigma_0\alpha^{\frac{p-1}{p}}}{B^{\frac{p-1}{p}}} & \leq & 2\tau_\star\sigma_0\cdot\tfrac{(1-\delta)\epsilon}{32\tau_\star\sigma_0}=\tfrac{(1-\delta)\epsilon}{16}, \\[4pt]
\tfrac{2\tau_\star\sigma_1L_0\eta}{B^{\frac{p-1}{p}}\alpha} & \leq & 2\cdot\tfrac{\sqrt{(1-\delta)L_2\epsilon}}{4}\cdot\tfrac18\sqrt{\tfrac{(1-\delta)\epsilon}{L_2}}=\tfrac{(1-\delta)\epsilon}{16}, \\[4pt]
\tfrac{2L_2\eta^2}{\alpha^2} & \leq & 2L_2\cdot\tfrac{(1-\delta)\epsilon}{64L_2}=\tfrac{(1-\delta)\epsilon}{32}.
\end{array}
\end{equation*}
Therefore, we have
\begin{equation*}
\mathbb E\|\nabla F(\widetilde X_T)\|_\star\leq\tfrac{32}{21(1-\delta)}\cdot\tfrac{6(1-\delta)\epsilon}{16}=\tfrac{4\epsilon}{7}<\epsilon.
\end{equation*}
It remains to count oracle calls. Since Algorithm~\ref{algorithm:transported-Muon} calls the stochastic gradient oracle $B$ times per iteration, the total number of calls is $BT$. By the definitions of $\alpha$ and $\eta$, we have
\begin{equation*}
\tfrac{B}{\alpha}=\max\left\{B,\left(\tfrac{32\tau_\star\sigma_0}{(1-\delta)\epsilon}\right)^{\frac{p}{p-1}},\tfrac{BL_0}{\sqrt{(1-\delta)L_2\epsilon}}\right\}, \quad \tfrac{1}{\eta}\leq\tfrac{8}{\alpha}\left(\tfrac{L_1}{1-\delta}+\sqrt{\tfrac{L_2}{(1-\delta)\epsilon}}\right).
\end{equation*}
Therefore, we have
\begin{equation*}
BT\leq2B+\tfrac{B}{\alpha}\left(S_\epsilon+D_\epsilon\right),\quad S_\epsilon:=\tfrac{128\Delta_0L_1}{(1-\delta)^2\epsilon}+\tfrac{32(\sigma_0+\sigma_1\|\nabla F(X_0)\|_\star)}{(1-\delta)\epsilon},\quad D_\epsilon:=\tfrac{128\Delta_0\sqrt{L_2}}{(1-\delta)^{3/2}\epsilon^{3/2}}.
\end{equation*}
Let $q:=\frac{p}{p-1}$ and $y_\epsilon:=(\frac{32\tau_\star\sigma_0}{(1-\delta)\epsilon})^q$. We show that the product $y_\epsilon D_\epsilon$ dominates for all $\epsilon$ below a threshold independent of $m$, $n$, and $\tau_\star$. Let $M_\epsilon:=\max\{1,\frac{8\sigma_1}{1-\delta},\frac{4\sigma_1L_0}{\sqrt{(1-\delta)L_2\epsilon}}\}$. Since $\tau_\star\geq1$ by Lemma~\ref{lemma:dimension-upper-bound}, we have $B\leq2\tau_\star^qM_\epsilon^q$, and hence
\begin{equation*}
\tfrac{B}{\alpha}\leq\tau_\star^q\max\left\{2M_\epsilon^q,\left(\tfrac{32\sigma_0}{(1-\delta)\epsilon}\right)^q,\tfrac{2M_\epsilon^qL_0}{\sqrt{(1-\delta)L_2\epsilon}}\right\}.
\end{equation*}
As $\epsilon\to0$, we have $M_\epsilon^q=O(\epsilon^{-\frac{q}{2}})$, and $\frac{q}{2}+\frac12<q$ since $q\geq2$. Hence, the first and the third entries in the above maximum are $o(\epsilon^{-q})$, $S_\epsilon=O(\epsilon^{-1})=o(D_\epsilon)$, and $M_\epsilon^q=o(\epsilon^{-q}D_\epsilon)$. Since none of these functions of $\epsilon$ depends on $m$, $n$, or $\tau_\star$, there exists $\epsilon_0>0$, depending only on $p,\delta,\sigma_0,\sigma_1,L_0,L_1,L_2,\Delta_0$, and $\|\nabla F(X_0)\|_\star$, such that, for all $\epsilon\in(0,\epsilon_0]$, we have $\frac{B}{\alpha}\leq y_\epsilon$, $S_\epsilon\leq D_\epsilon$, and $2B\leq4\tau_\star^qM_\epsilon^q\leq y_\epsilon D_\epsilon$. Therefore, we have
\begin{equation*}
BT\leq3y_\epsilon D_\epsilon=384\cdot32^{\frac{p}{p-1}}\tau_\star^{\frac{p}{p-1}}\sigma_0^{\frac{p}{p-1}}\sqrt{L_2}\Delta_0(1-\delta)^{-\frac{5p-3}{2p-2}}\epsilon^{-\frac{5p-3}{2p-2}}.
\end{equation*}
For $\|\cdot\|=\|\cdot\|_\textnormal{op}$ and $\|\cdot\|_\star=\|\cdot\|_\textnormal{nuc}$, Lemma~\ref{lemma:dimension-upper-bound} gives $\tau_\star^{\frac{p}{p-1}}=\Theta(\min\{m,n\})$, which proves the desired result. 
\end{proof}
\section{Setup and Additional Experiments}
\label{sec:app-exp}
\subsection{Experimental setup}
\paragraph{LLM experiment setup.} We use the tokenizer~\citep{nanochat} with a vocabulary of 32,768 provided by nanochat. The token embeddings and language-model head are untied. The architecture uses rotary embeddings~\citep{Su-2024-Roformer}, RMSNorm~\citep{Zhang-2019-Root}, QK normalization~\citep{Henry-2020-Query}, local attention~\citep{Beltagy-2020-Longformer} implemented with FlashAttention2~\citep{Dao-2024-Flashattention}, ReLU$^2$ activations, and value embeddings~\citep{Zhou-2025-Value}. The validation set split contains 42M tokens. Table~\ref{tab:app-model-configs} lists the model configurations and training budgets. Following nanochat, the training token count is eight times the number of transformer-matrix and language-model-head parameters.

The LLM experiments use three nodes with ten A40 GPUs each. One node is slightly slower than the other two. Across multiple runs on all three nodes, Algorithm~\ref{algorithm:transported-Muon} averages 19.8 and 81.9 minutes for the 287M and 539M models, respectively; the other Muon-family methods take 19.2--19.5 and 80.2--80.7 minutes. For the 1.39B model on a faster node, Algorithm~\ref{algorithm:transported-Muon} takes 10.46 hours, compared with 10.26 hours for each of the heavy-ball and Nesterov baselines.

\begin{table}[!t]
\centering
\caption{Nanochat model configurations.}
\label{tab:app-model-configs}
\begin{tabular}{|l|c|c|c|c|c|c|c|c|c|}\hline
Model & Depth & Width & Heads & Size & Tokens/step & Steps & Train tokens & $s_B$ & $s_d$\\\hline
287M  & 12 & 768  & 6  & 287M   & 491,520 & 1,794 & 882M & 0.9682 & 1.0000 \\\hline
539M  & 16 & 1024 & 8  & 539M  & 491,520 & 3,826 & 1.88B & 0.9682 & 0.8660\\\hline
1.39B & 24 & 1536 & 12 & 1.39B & 983,040 & 5,941 & 5.84B & 1.3693 & 0.7071 \\\hline
\end{tabular}%
\end{table}
\begin{table}[!t]
\centering
\caption{Optimizer settings for the parameter groups of nanochat.}
\label{tab:app-optimizer-groups}
\begin{tabular}{|l|l|l|c|c|}\hline
Parameter group & Optimizer & Peak learning rate & Betas or momentum & Weight decay \\\hline
Transformer matrices & AdamW & $\eta_{\mathrm{mat}}s_Bs_d$ & $(0.9,0.999)$ & scheduled \\\hline
Transformer matrices & Muon family & $\eta_{\mathrm{mat}}s_B$ & tuned & scheduled \\\hline
LM head & AdamW & $0.008s_Bs_d$ & $(0.8,0.96)$ & 0.01 \\\hline
Token embeddings & AdamW & $0.3s_Bs_d$ & $(0.8,0.995)$ & 0.001 \\\hline
Value embeddings & AdamW & $0.15s_Bs_d$ & $(0.8,0.995)$ & 0.01 \\\hline
\end{tabular}%
\end{table}

Each transformer block contains query, key, value, output, MLP-up, and MLP-down matrices. For model width $d$, the attention matrices are $d\times d$, while the MLP matrices are $4d\times d$ and $d\times4d$. Only these matrices switch between AdamW and Muon-family optimizers, and all other parameters use AdamW following standard practices~\citep{Jordan-2024-Muon}. Let $B$ denote the batch size measured by tokens. We define the batch-size scaling factor $s_B=\sqrt{B/2^{19}}$ and the AdamW width factor $s_d=(d_{\mathrm{model}}/768)^{-1/2}$. Table~\ref{tab:app-optimizer-groups} gives the learning rates for each parameter group. The learning rate has a 40-step linear warmup, a constant phase, and a linear decay over the final 65\% of training to a multiplier of 0.05. For transformer matrices, weight decay starts at $\lambda = \lambda_{\mathrm{ref}}\sqrt{B/B_{\mathrm{ref}}}(D_{\mathrm{ref}}/D)$ and follows a cosine schedule to zero, where $D$ is the training token budget, $D_{\mathrm{ref}}$ is the corresponding 287M model's training token budget for reference, $\lambda_{\mathrm{ref}}=0.28$, and $B_{\mathrm{ref}}=2^{19}$.

We select hyperparameters by a grid search using validation loss as the metric, training each candidate configuration with one seed. For AdamW, we fix the betas (first and second moment factors) at $(0.9,0.999)$ and sweep the learning rate on the 287M and 539M models over $\{0.0006,\allowbreak 0.0008, \allowbreak 0.0010,\allowbreak 0.0012, \allowbreak 0.0015,\allowbreak 0.0018, \allowbreak 0.0022\}$. The best value, 0.0012, is used for both the 287M and 539M models. For the 1.39B model, a learning rate of 0.0010 gives lower validation loss than that of 0.0012, and hence we use 0.001. For Muon-family methods, we first sweep the learning rate over $\{0.0025,\allowbreak 0.005,\allowbreak 0.01,\allowbreak 0.02,\allowbreak 0.03,\allowbreak 0.04,\allowbreak 0.05\}$ on the 287M and 539M models with heavy-ball momentum 0.90 and Nesterov momentum 0.95. Performance is robust, with optima typically between 0.02 and 0.03. At the selected learning rates, we sweep heavy-ball and Nesterov momentum separately over $\{0.90,\allowbreak 0.91,\allowbreak 0.92,\allowbreak 0.93,\allowbreak 0.94,\allowbreak 0.95\}$ for the 287M model. For the 539M model, we sweep heavy-ball momentum over $\{0.88,\allowbreak 0.89, \allowbreak 0.90, \allowbreak 0.91, \allowbreak0.92, \allowbreak0.93\}$ and Nesterov momentum over $\{0.91, \allowbreak0.92, \allowbreak0.93, \allowbreak0.94, \allowbreak0.95, \allowbreak0.96\}$. For Algorithm~\ref{algorithm:transported-Muon} with Eq.~\eqref{eq:TUSCG-alpha}, we use the selected heavy-ball momentum values and sweep $\gamma$ over $\{0.005,0.0075,0.01\}$. On the 287M and 539M models, validation loss worsens substantially for $\gamma\geq 0.01$. We do not perform a grid search for the 1.39B models, and we take its learning rates and momentum values from the 539M model for the Muon variants. 

Muon~\citep{Jordan-2024-Muon} uses Nesterov momentum in Eq.~\eqref{eq:nesterov-momentum}. In every Muon-family method, each LMO normalizes its input by the Frobenius norm and applies five PolarExpress iterations~\citep{Amsel-2026-Polar} in bfloat16. For an $m\times n$ update, NorMuon~\citep{Li-2025-Normuon} keeps an exponential moving average, with $\beta_2=0.9$, of the mean squared entry of each row if $m\geq n$ and of each column otherwise. NorMuon and AdamW use stabilization $\varepsilon=10^{-10}$. For Algorithm~\ref{algorithm:transported-Muon}, we store the previous displacement $X_t-X_{t-1}$ of the transformer matrices, shift these matrices by $\gamma(X_t-X_{t-1})$ for the forward and backward passes, and restore them before the optimizer update. Other parameters are not shifted, although their gradients are evaluated at the shifted matrices. The CNN experiments use the same implementation.

\begin{table}[!t]
\centering
\caption{Selected hyperparameters for LLM experiment. The table reports the base learning rate $\eta_{\mathrm{mat}}$ before the scaling in Table~\ref{tab:app-optimizer-groups}. For scheduled Nesterov momentum, the momentum increases linearly from $0.85$ to $0.97$ over the first 400 steps, remains at $0.97$, and decreases to $0.90$ during the final learning rate decay.}
\label{tab:app-llm-selected-hparams}
\begin{tabular}{|l|l|c|c|c|}\hline
Model & Optimizer & \shortstack{Matrix\\learning rate} & \shortstack{Betas or\\momentum} & Transport $\gamma$ \\\hline
287M & AdamW & 0.0012 & $(0.9,0.999)$ & -- \\\hline
287M & Muon & 0.03 & 0.94 & -- \\\hline
287M & NorMuon  & 0.02 & 0.93 & -- \\\hline
287M & NorMuon with Eq.~\eqref{eq:nesterov-momentum} & 0.03 & 0.94 & -- \\\hline
287M & NorMuon (scheduled N. mom.) & 0.02 & $0.85\to0.97\to0.90$ & -- \\\hline
287M & Algorithm~\ref{algorithm:transported-Muon} with Eq.~\eqref{eq:TUSCG-alpha}  & 0.02 & 0.93 & 0.005 \\\hline
539M & AdamW  & 0.0012 & $(0.9,0.999)$ & -- \\\hline
539M & Muon & 0.02 & 0.95 & -- \\\hline
539M & NorMuon  & 0.02 & 0.90 & -- \\\hline
539M & NorMuon with Eq.~\eqref{eq:nesterov-momentum}  & 0.02 & 0.95 & -- \\\hline
539M & NorMuon (scheduled N. mom.) & 0.02 & $0.85\to0.97\to0.90$ & -- \\\hline
539M & Algorithm~\ref{algorithm:transported-Muon} with Eq.~\eqref{eq:TUSCG-alpha}  & 0.02 & 0.90 & 0.0075 \\\hline
1.39B & AdamW & 0.0010 & $(0.9,0.999)$ & -- \\\hline
1.39B & Muon  & 0.02 & 0.95 & -- \\\hline
1.39B & NorMuon & 0.02 & 0.90 & -- \\\hline
1.39B & NorMuon with Eq.~\eqref{eq:nesterov-momentum}  & 0.02 & 0.95 & -- \\\hline
1.39B & NorMuon (scheduled N. mom.) & 0.02 & $0.85\to0.97\to0.90$ & -- \\\hline
1.39B & Algorithm~\ref{algorithm:transported-Muon} with Eq.~\eqref{eq:TUSCG-alpha}  & 0.02 & 0.90 & 0.010 \\\hline
\end{tabular}
\end{table}

\paragraph{Downstream performance.} For the 1.39B models, we report accuracy on $10$ knowledge and reasoning benchmarks (Table~\ref{tab:app-nanochat-1b-downstream}) following OLMES~\citep{Gu-2025-OLMES}: MMLU~\citep{Hendrycks-2021-Measuring}, HellaSwag~\citep{Zellers-2019-Hellaswag}, PIQA~\citep{Bisk-2020-Piqa}, WinoGrande~\citep{Sakaguchi-2021-Winogrande}, ARC-Challenge (ARC-C) and ARC-Easy (ARC-E)~\citep{Clark-2018-Think}, BoolQ~\citep{Clark-2019-BoolQ}, CommonsenseQA (CSQA)~\citep{Talmor-2019-CommonsenseQA}, Social IQa (SIQA)~\citep{Sap-2019-Social}, and OpenBookQA (OBQA)~\citep{Mihaylov-2018-Can}. NorMuon variants and  Algorithm~\ref{algorithm:transported-Muon} with Eq.~\eqref{eq:TUSCG-alpha} (transported Muon) have comparable performance across baselines.

\begin{table}[!t]
\centering
\caption{Validation loss and downstream performance of the 1.39B nanochat models.}
\label{tab:app-nanochat-1b-downstream}
\begingroup
\small
\setlength{\tabcolsep}{5pt}
\renewcommand{\arraystretch}{1.15}
\begin{tabular}{|l|c|c|c|c|c|c|}
\hline
Metric & AdamW & Muon
& \shortstack{NorMuon\\(heavy-ball)}
& \shortstack{NorMuon\\(Nesterov)}
& \shortstack{NorMuon\\(scheduled)}
& \shortstack{Transported\\Muon}\\
\hline
Val. loss & 2.4456 & 2.3724 & 2.3526 & 2.3543 & \textbf{2.3476} & \underline{2.3525} \\
\hline
MMLU & 30.21 & 30.97 & 31.34 & \underline{31.60} & \textbf{31.83} & 31.53 \\
HellaSwag & 50.33 & 54.67 & \textbf{56.41} & 55.70 & \underline{56.37} & 56.17 \\
PIQA & 72.69 & 73.18 & \textbf{74.92} & 74.54 & 74.27 & \underline{74.59} \\
WinoGrande & 52.72 & 53.59 & \textbf{55.64} & 54.93 & \underline{55.01} & \underline{55.01} \\
ARC-C & 41.30 & 40.10 & 41.98 & \underline{43.77} & 43.26 & \textbf{45.65} \\
ARC-E & 67.63 & 71.25 & \textbf{72.69} & \underline{71.46} & 71.42 & 71.34 \\
BoolQ & 60.55 & 61.93 & 58.47 & \underline{63.09} & 61.44 & \textbf{63.39} \\
CSQA & 53.73 & \textbf{57.17} & 56.27 & 55.77 & \underline{56.67} & \textbf{57.17} \\
SIQA & 46.93 & 47.19 & \textbf{49.13} & 47.49 & \underline{48.57} & 47.39 \\
OBQA & 44.40 & \underline{46.00} & 44.20 & 45.60 & \textbf{48.40} & 45.80 \\
\hline
Downstream avg. & 52.05 & 53.60 & 54.10 & 54.40 & \underline{54.72} & \textbf{54.80} \\
\hline
\end{tabular}
\endgroup
\end{table}

\paragraph{CNN setup.} Besides whitening layers and a linear classifier head, the CIFARNET model has three convolutional groups of widths $64,256,256$. Each group applies a $3\times3$ convolution, max pooling, BatchNorm, GELU, a second $3\times3$ convolution, BatchNorm, and GELU. Flattening each convolution into $\mathbb R^{d_\textnormal{out}\times9d_\textnormal{in}}$ gives six matrices: one matrix of each size $64\times216$, $64\times576$, and $256\times576$, plus three of size $256\times2304$. We compare the baseline optimizers on these matrices. For every compared method, the remaining parameters (the whitening bias, the BatchNorm biases, and the classifier head) use SGD with Nesterov momentum $0.989703$ and no dampening, with peak learning rates $1.4949\times10^{-3}$ for the biases and $1.72446$ for the head, following \citet{Jordan-2024-Single}.

We train for 50 epochs with batch size 512, label smoothing 0.2, and no weight decay. All parameter groups use a 5\% linear warmup followed by cosine decay. Before momentum accumulation, we clip the global gradient norm of all trainable parameters at 1.0. For Algorithm~\ref{algorithm:transported-Muon}, this gradient is evaluated at the shifted convolutional matrices. We apply random horizontal flips and translations of up to two pixels with reflection padding. For hyperparameter selection, we split the official training set into 45,000 training and 5,000 validation examples. We then retrain each selected configuration on all 50,000 examples and evaluate it on the official test set.

We select CIFARNET hyperparameters based on mean final validation accuracy across 20 tuning seeds. For AdamW, we sweep learning rates for the convolutional layers over $\{0.0003,\allowbreak 0.001,\allowbreak 0.003,\allowbreak 0.01\}$. For heavy-ball SGD and Muon-family methods, we sweep learning rates $\{0.01,\allowbreak 0.02,\allowbreak 0.03,\allowbreak 0.04,\allowbreak 0.05,\allowbreak 0.075,\allowbreak 0.10,\allowbreak 0.15,\allowbreak 0.20,\allowbreak 0.24\}$ and momentum values $\{0.5,\allowbreak 0.6,\allowbreak 0.7,\allowbreak 0.8,\allowbreak 0.9,\allowbreak 0.95\}$. We report results on the official test set over 20 seeds not used for tuning. Table~\ref{tab:app-cnn-selected-hparams} lists the selected hyperparameters.

\begin{table}[!t]
\centering
\caption{Selected CIFARNET hyperparameters.}
\label{tab:app-cnn-selected-hparams}
\begin{tabular}{|l|l|c|c|c|}\hline
CIFAR- & Optimizer & Learning rate & Betas or momentum & Transport $\gamma$ \\\hline
10 & AdamW  & 0.003 & $(0.9,0.999)$ & -- \\\hline
10 & SGD (heavy-ball) & 0.24 & 0.95 & -- \\\hline
10 & Muon & 0.03 & 0.90 & -- \\\hline
10 & NorMuon & 0.10 & 0.70 & -- \\\hline
10 & NorMuon with Eq.~\eqref{eq:nesterov-momentum} & 0.10 & 0.60 & -- \\\hline
10 & Algorithm~\ref{algorithm:transported-Muon} & 0.03 & 0.90 & 9 \\\hline
100 & AdamW & 0.01 & $(0.9,0.999)$ & -- \\\hline
100 & SGD (heavy-ball) & 0.24 & 0.95 & -- \\\hline
100 & Muon & 0.10 & 0.80 & -- \\\hline
100 & NorMuon & 0.15 & 0.50 & -- \\\hline
100 & NorMuon with Eq.~\eqref{eq:nesterov-momentum}  & 0.10 & 0.80 & -- \\\hline
100 & Algorithm~\ref{algorithm:transported-Muon} & 0.05 & 0.80 & 4 \\\hline
\end{tabular}%
\end{table}

\subsection{Additional experiments}
\label{sec:noise-validation}
We evaluate the gradient noise and LMO approximation errors in CIFARNET and nanochat and compare them with Assumptions~\ref{assumption:noise} and~\ref{assumption:inexactness}.

\paragraph{Gradient noise.}
At each saved checkpoint of the CNN and LLM experiments, we compute a full-batch gradient and the mini-batch gradient noise for each matrix in FP32, without clipping or data augmentation and with BatchNorm in evaluation mode. Following \citet{Simsekli-2019-Tail}, we use the estimator of \citet{Mohammadi-2015-Estimating} to estimate a tail index from the coordinates of the noise. At a saved model state $X$, we split an evaluation set of $N$ examples or token sequences into $M=N/B$ disjoint mini-batches of size $B$. For a parameter matrix $W$, let $G_W(X,\xi^{(j)})$ be the gradient of mini-batch $j$. The full-batch gradient over the evaluation set and the centered mini-batch noise are
\begin{equation*}
\overline G_W(X)=\tfrac{1}{M}\sum_{j=1}^M G_W(X,\xi^{(j)}),\quad U_W^{(j)}=G_W(X,\xi^{(j)})-\overline G_W(X).
\end{equation*}
For each $W$, we flatten $U_W^{(1)},\ldots,U_W^{(M)}$ and concatenate their coordinates into a scalar sample $x_1,\ldots,x_K$. We choose the divisor $K_1$ of $K$ closest to $\sqrt K$, set $K_2=\lfloor K/K_1\rfloor$, retain the first $K_1K_2$ coordinates, and form $y_i=\sum_{j=1}^{K_1}x_{j+(i-1)K_1}$ for $i=1,\ldots,K_2$. The tail index estimate is
\begin{equation*}
\widehat a=\left\{\tfrac{1}{\log K_1}\left(\tfrac{1}{K_2}\sum_{i=1}^{K_2}\log|y_i|-\tfrac{1}{K_1K_2}\sum_{i=1}^{K_1K_2}\log|x_i|\right)\right\}^{-1}.
\end{equation*}
We report $\widehat a$ as a descriptive tail-index estimate for the pooled gradient coordinates.

For CIFARNET, we use all 50,000 training examples, split into 100 mini-batches of 500 for 20 independently trained models. For the 287M nanochat model, we estimate the tail index separately for the Q, K, V, O, MLP-up, and MLP-down matrices in each transformer block. At all $8$ checkpoints, we use the same 8,192 token sequences, split into 32 mini-batches of 256. 

\begin{figure}[t]
\centering
\includegraphics[width=\linewidth]{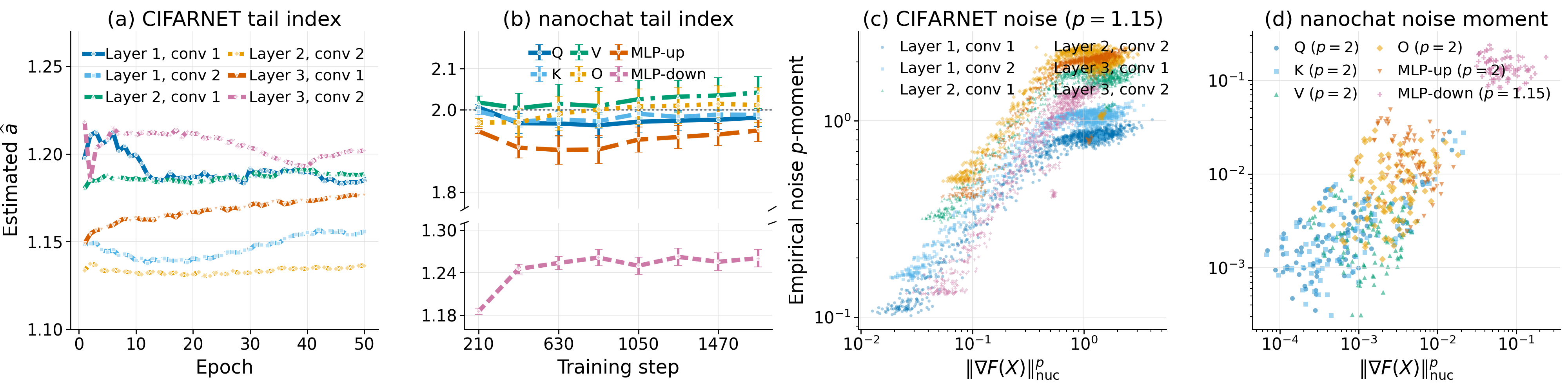}
\caption{Empirical gradient noise measurements. (a) tail-index estimates during training for CIFARNET (b) tail-index estimates during training for nanochat (c) mini-batch noise moment for CIFARNET (d) mini-batch noise moment for nanochat.}
\label{fig:noise-measurements}
\end{figure}

For CIFARNET, the estimated tail indices lie between approximately $1.12$ and $1.25$ (Figure~\ref{fig:noise-measurements}(a)). Each nanochat transformer block contains the usual Q, K, V, O, MLP-up, and MLP-down matrices. We estimate the tail indices on these matrices and report the standard error across different matrices in the same group in Figure~\ref{fig:noise-measurements}(b). The QKVO gradients have estimated tail indices near $2$. The MLP-up gradients are slightly heavier-tailed, and the MLP-down gradients have tail index estimates between approximately $1.18$ and $1.27$ (Figure~\ref{fig:noise-measurements}(b)). Thus, tail behavior varies across parameters. 

Figures~\ref{fig:noise-measurements}(c) and~\ref{fig:noise-measurements}(d) plot $\frac{1}{M}\sum_{j=1}^M\|U_W^{(j)}\|_{\textnormal{nuc}}^p$ against $\|\overline G_W(X)\|_{\textnormal{nuc}}^p$ on logarithmic scales. We use $p=1.15$ for CIFARNET and the nanochat MLP-down matrices and $p=2$ for the remaining nanochat matrices for illustration. The noise moment generally increases with the gradient norm, which qualitatively motivates the additive and multiplicative terms in Assumption~\ref{assumption:noise}. For some parameter groups, the noise norm is nearly constant, which suggests that $\sigma_1=0$ may be proper for some cases. 

\paragraph{Performance of inexact subroutines.}
Normalization can enforce the condition $\|\operatorname{lmo}_\delta(X)\|\leq 1$ in Assumption~\ref{assumption:inexactness}. Indeed, for $h:\mathbb{R}^{m\times n}\to\mathbb{R}^{m\times n}$ with $h(0)=0$ and $0<C_h:=\sup_X\|h(X)\|<+\infty$, let $\kappa_h:=\inf_{X\neq0}\frac{-\langle X,h(X)\rangle}{C_h\|X\|_\star}$. Then $h/C_h$ satisfies Assumption~\ref{assumption:inexactness} for some $\delta<1$ if and only if $\kappa_h>0$, and the smallest inexactness parameter is $\delta_h=1-\kappa_h$.

We compare classical Newton--Schulz, Muon's Newton--Schulz implementation with improved coefficients~\citep{Jordan-2024-Muon}, and PolarExpress on momentum matrices from a 287M nanochat run. For a momentum matrix $M\neq0$ and the output $P$ of an inexact LMO applied to $M$, we measure the pointwise approximation error
\begin{equation*}
\delta_M:=1-\tfrac{-\langle M,P\rangle}{\max\{\|P\|_{\textnormal{op}},1\}\|M\|_{\textnormal{nuc}}}\in[0,1].
\end{equation*}
The factor $\max\{\|P\|_{\textnormal{op}},1\}$ maps $P$ into the unit ball. Unlike $\delta_h$, $\delta_M$ is computed on the encountered inputs only and empirically evaluates Assumption~\ref{assumption:inexactness}. Figure~\ref{fig:matrix-sign-delta}(a) shows $\delta_M$ versus the number of inner iterations, averaged over all matrix-valued parameters. With five inner iterations, Figure~\ref{fig:matrix-sign-delta}(b) shows $\delta_M$ at each checkpoint, averaged over all matrix-valued parameters, and Figure~\ref{fig:matrix-sign-delta}(c) shows $\delta_M$ for each matrix type. In~(a) and~(c), we average $\delta_M$ over the matrices at each of the $8$ checkpoints and report the mean and standard error of these averages. PolarExpress with five inner iterations attains the smallest mean error and is robust across training steps and matrix types.
\begin{figure}[ht!]
\centering
\includegraphics[width=\linewidth]{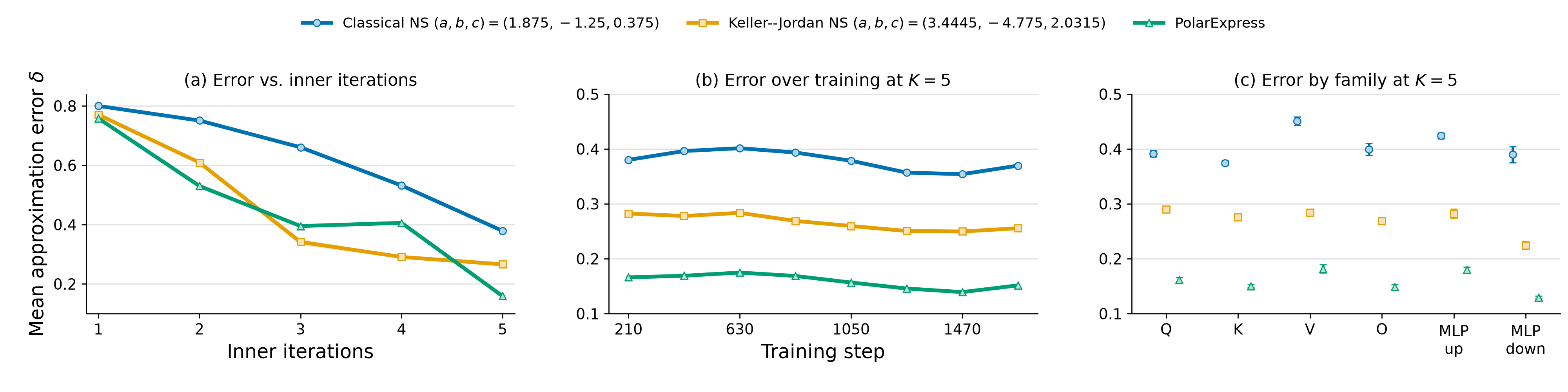}
\caption{Empirical LMO approximation error for nanochat momentum matrices.}
\label{fig:matrix-sign-delta}
\end{figure}

\end{document}